\documentclass[UTF8,10pt]{article}
\usepackage[left=1.91cm,right=1.91cm,top=2.54cm,bottom=2.54cm]{geometry}
\usepackage{amsfonts}
\usepackage{amsmath}
\usepackage{amssymb}
\usepackage{pgfplots}
\usepackage{tikz}
\usetikzlibrary{arrows}
\usepackage{titlesec}
\usepackage{bm}
\usepackage{appendix}
\usepackage{tikz-cd}
\usepackage{mathtools}
\usepackage{amsthm}
\usepackage{enumitem}
\setenumerate[1]{itemsep=0pt,partopsep=0pt,parsep=\parskip,topsep=5pt}
\setitemize[1]{itemsep=0pt,partopsep=0pt,parsep=\parskip,topsep=5pt}
\setdescription{itemsep=0pt,partopsep=0pt,parsep=\parskip,topsep=5pt}
\usepackage[colorlinks,linkcolor=black,anchorcolor=black,citecolor=black]{hyperref}
\usepackage{setspace}
\usepackage[numbers]{natbib}
\usepackage{cases}
\usepackage{fancyhdr}
\usepackage{txfonts}

\newtheorem{theorem}{Theorem}[section]
\newtheorem{proposition}{Proposition}[section]
\newtheorem{lemma}{Lemma}[section]
\newtheorem{corollary}{Corollary}[section]
\newtheorem{remark}{Remark}[section]

\numberwithin{equation}{section}

\numberwithin{equation}{section}
\usepackage{xcolor}

\begin{document}

\bibliographystyle{plain}
\title{\textbf{
 GLOBAL AND LOCAL THEORY OF SKEW MEAN CURVATURE
FLOWS  }}

\author{Ze Li\thanks{School of Mathematics and Statistics, Ningbo University, Ningbo, 315211, Zhejiang, P.R. China. Email: \texttt{rikudosennin@163.com}} }
\date{ }
 \maketitle



\begin{abstract}
In this paper, we study  the skew mean curvature flow.  The results are threefold. First, we prove the global regularity of   solutions with  initial data which are small perturbations of planes in Sobolev spaces.  Second, we prove  the modified scattering and the existence of wave operators for small data, which completely determines the set of asymptotic states.  Third, we study the Cauchy problem for arbitrary large data.
\end{abstract}


\section{Introduction}

In view of geometric flows, the skew mean curvature flow (SMCF) evolves a codimension 2 submanifold along its binormal direction with a speed given by its
mean curvature. Precisely speaking, assume that $\Sigma$  is an $n$-dimensional oriented manifold and $(\mathcal{M},h)$ is an  $(n+2)$-dimensional oriented Riemannian manifold, then
SMCF is  a family of time-dependent immersions $F :{\Bbb I}\times \Sigma \to \mathcal{M}$, which satisfies
\begin{align}\label{1}
\left\{
  \begin{array}{ll}
     {\partial_t}F= J(F)\mathbf{H}(F), \mbox{ }(t,x)\in \Bbb I\times \Sigma& \hbox{ } \\
    F(0,x)=F_0(x) & \hbox{}
  \end{array}
\right.
\end{align}
where for each given $t\in {\Bbb I}$, $\mathbf{H}(F)$ denotes the mean curvature vector
of the submanifold $\Sigma_t:= F(t,\Sigma)$. Here, $J(F)$, which denotes the  natural induced complex structure for the normal bundle $N\Sigma_t$,  can be simply defined as rotating a vector in the normal space by $\frac{\pi}{2}$ positively (notice that $N\Sigma_t$ is of rank 2).

For $n=1$, the 1-dimensional SMCF in  $\Bbb R^3$ is the vortex filament equation
$
\partial_t v= \partial_sv\times \partial^2_sv
$
for $v:(s,t)\in \Bbb R\times \Bbb R\longmapsto v(s,t)\in \Bbb R^3$, where $t$ denotes time, $s$ denotes the arc-length parameter of the curve $v(t,\cdot)$,  and $\times$ denotes
the cross product in $\Bbb R^3$. The  vortex filament equation describes  the free motion
of a vortex filament, see Da Rios \cite{RGG}, Hasimoto \cite{HGGa}.
For $n\ge 2$, SMCF was deduced by both physicists and  mathematicians. The physical  motivations  are   the localized
induction approximation (LIA) of high dimensional  Euler equations and  asymptotic dynamics of vortices in
superconductivity and superfluidity, see Lin \cite{LGG1}, Jerrard \cite{jGGd}, Shashikanth \cite{SGGh}, Khesin \cite{KGG}, Arnold-Khesin \cite{AK}.
SMCF  also appears in various mathematical problems, see  Terng \cite{TGGe}, Marsden-Weinstein \cite{MW},  Haller-Vizman \cite{HV}.

Let us recall the non-exhaustive list of works on SMCF for $n\ge 2$. Song-Sun \cite{SGGoGGnGGgGG1}  proved local existence of SMCF for $F:\Sigma \to \Bbb R^{4}$ with compact oriented surface $\Sigma$. This was generalized by Song \cite{SGGoGGnGGgGG2} to $F:\Sigma \to \Bbb R^{n+2}$ with compact oriented surface $\Sigma$ for all $n\ge 2$, and \cite{SGGoGGnGGgGG2} proved uniqueness and  continuous dependence of solution on  initial data for  arbitrary oriented manifold $\Sigma$. Khesin-Yang \cite{KGGYgghh} constructed an example showing that the SMCF can blow up in finite time. In the PDE view,  Gomez \cite{GGhhh} proposed a way to write SMCF as a quasilinear Schr\"odinger equation system different from ours. For the small data global theory and local well-posedness theory  on general  quasilinear Schr\"odinger equations, see the pioneer  works of \cite{KP} and \cite{Ke1,Ke2,Ke3,MMT,MMT1} respectively.

It is remarkable that SMCF has a deep relationship with Schr\"odinger map flow (e.g. \cite{UT}), in fact,
\cite{SGGoGGnGGgGG3} proved that the Gauss map of an $m$-dimensional SMCF in $\Bbb R^{m+2}$ fulfills a Schr\"odinger map flow equation. This property also plays an important role in the proof of uniqueness and continuous  dependence on initial data of \cite{SGGoGGnGGgGG2}.

The small data global regularity  problem for SMCF was initiated in our previous work \cite{Li} where we proved Euclidean planes are stable under SMCF for small transversal perturbations in some $W^{2,q}\cap H^k$ space  with some $q\in(1,2)$.
In this work, our first aim is to remove the $W^{2,q}$ smallness and transversal assumption of \cite{Li} when $n\ge 3$.
Our second aim is to provide a  complete picture of long time dynamics of small solutions, especially determining what the set of asymptotic states is.
The main theorem is as follows.

Let $E:\Bbb R^{n}\to \Bbb R^{n+2}$ denote the plane
\begin{align}\label{2dVBn}
E(x)=(x_1,...,x_n,{\xi}_1\cdot x+a_1,{\xi}_2\cdot x+a_2),
\end{align}
where $\{{\xi}_i\}^{2}_{i=1}\subset\Bbb R^n$ are given constant vectors, and $a_1,a_2\in \Bbb R$ are given constants. Let ${H}^{\sigma}_{E}$ be maps $F:\Bbb R^n\to \Bbb R^{n+2}$ which satisfy
\begin{align*}
\|F(x) -E(x)\|_{ H^{\sigma}(\Bbb R^{n})} <\infty.
\end{align*}

\begin{theorem}\label{T2}
Assume that $n\ge 3$, $k \in \Bbb Z_+$,   $  k\ge n+4$. There exists a sufficiently small constant $\epsilon_*>0$  such that if $F_0\in  {H}^{{k}}_{E}$ satisfies
\begin{align}\label{dVBn}
\|F_0(x)-E(x)\|_{ H^{k}} \le \epsilon_*,
\end{align}
then there exists a global unique solution $F\in C([-T,T]; H^{{k}}_{E})$ for all $T>0$  to SMCF with initial data $F_0$. Moreover, there exist an orthogonal transformation ${\Bbb B}\in O(n+2)$, a constant vector $\vec{e}\in \Bbb R^{n+2}$, a sufficiently smooth homeomorphism $\tilde{\Phi}:\Bbb R^{n}\to \Bbb R^{n}$, and time independent complex valued  functions $\phi_{\pm}\in H^{k}$ for which
\begin{align*}
{\Bbb B} {F}_0(\tilde{\Phi}(x))-\vec{e}=(x,0,0),\mbox{ }\forall x\in \Bbb R^n,
\end{align*}
 and
\begin{align}
&\lim\limits_{t\to \pm \infty}\|  ({\Bbb B F-\vec{e}})^{n+1}(t,\tilde{\Phi}(x)) + \mathrm{i}({\Bbb B} F-\vec{e})^{n+2}(t,\tilde{\Phi}(x)) -e^{\mathrm{i}t\Delta}\phi_{\pm} \|_{H^{k_n}}=0,\label{scatter}
\end{align}
where $k_n$ is the smallest integer that is bigger than or equal to  $\frac{n}{2}+1$. Furthermore, if $F_0\in \cap_{\sigma\ge 0} H^{\sigma}_{E}$, then $F(t)$ is  smooth and  belongs to $C([-T,T]; H^{{\sigma}}_{E})$ for all $T>0,\sigma\ge 0$.

Moreover, there exists $\varepsilon_*>0$ sufficiently small so that the following holds:  For any given  $\phi_{\infty}\in H^{k}$ with $\|\phi_{\infty}\|_{H^{k}}\le \varepsilon_*$, there exists initial data $F_0$ satisfying
$$\|F_0-(x,0,0)\|_{{H}^{k}}\le C\varepsilon_*,$$
 so that the corresponding solution of SMCF fulfills
 \begin{align}
\lim\limits_{t\to \infty}\|  F^{n+1}(t,x) + \mathrm{i}F^{n+2}(t,x) -e^{\mathrm{i}t\Delta}\phi_{\infty} \|_{H^{k_n}}=0.\label{scatter2}
\end{align}
\end{theorem}

\begin{remark}
Theorem \ref{T2} indeed includes three results, i.e., the global regularity of SMCF with  small data in Sobolev spaces, the long time behaviors of the solution, and the existence of wave operators.
\end{remark}

\begin{remark}
Theorem \ref{T2} still  holds with (\ref{dVBn}) replaced by $\|F_0-E\|_{\dot{H}^1\cap \dot{H}^{k}}\le \epsilon_*$, $F_0-E\in L^2$. The main arguments of this paper with more detailed analysis could low down the regularity index  $k$ of Theorem \ref{T2}, but it seems that some basic new ideas are needed for the critical regularity index $\frac{n}{2}+1$.
\end{remark}

Let us describe the main idea in proving  Theorem \ref{T2}. The whole proof is divided into four steps. The first step is to choose suitable coordinates for $ \Sigma$ and appropriate frames for $N\Sigma_t$  to simplify the equation. The immersions $\{F(t)\}$ under the new coordinates and frames are distorted to be a complex valued function $\phi(t,x)$, which satisfies a quasilinear Schr\"odinger equation.  Second, applying Keel-Tao's endpoint Strichartz estimates to the quasilinear Schr\"odinger equation of $\phi$ gives spacetime bounds for the second fundamental form. Third,   spacetime bounds and  energy estimates of the second fundamental form yield its uniform Sobolev norm bounds. Forth, the Sobolev norms of the map $F$ can be then controlled by that of  second fundamental forms, and thus proving uniform bounds of Sobolev norms of $F$. This road map with a bootstrap argument finishes the proof of global regularity. The long time dynamics and existence of wave operators can be  reduced to spacetime bounds of  $F$ as well.
In the view of PDE, the proof here is an endpoint refinement of Klainerman-Ponce \cite{KP}'s method.

\cite{SGGoGGnGGgGG2} raised the open problem of proving global theory of SMCF with certain small energy. Our previous work \cite{Li} studied  this problem for initial data which are small perturbations of Euclidean planes in non-$L^2$ Sobolev spaces. From the geometric view,  it is natural to use $L^2$ Sobolev norms of second fundamental forms as the realization  of ``certain small energy". Theorem 1.1 considered  initial data of small $H^k$ norms, and we remark that this is  indeed almost equivalent to studying data of small second fundamental forms in Sobolev spaces up to  differmorphisms in $\Sigma$. On the other side, as far as we know very few quasilinear dispersive PDEs have small data global theory in $H^k$ spaces. In fact,  most results assume data are small in weighted Sobolev spaces or in   non-$L^2$ Sobolev spaces.  Here in Theorem 1.1, for the highly quasilinear SMCF, we only assume the data are small in $H^k$ spaces.

In  Section 7, we study the local Cauchy problem for arbitrary large data in $H^{\sigma}_{E}$, see Theorem 7.1 and Corollary \ref{z222}.
The local Cauchy problem for arbitrary large data is not trivial since SMCF is a  quasilinear equation. Moreover, to prove well-posedness for large data it is generally necessary to impose Mizohata's integrability condition(\cite{M1,M2,M3}). For SMCF,  the geometric structure of the flow and geometric energy methods enable us to compensate the derivative loss, and no additional integrability conditions are needed. In addition, we study  not only the intrinsic geometric quantities but also norms of the extrinsic immersion $F$, which largely simplifies  the compactness argument and   provides the existence of strong solutions besides classic solutions.

The paper is organized as follows. In Section 2, we reduce SMCF to a quasilinear Schr\"odinger equation by choosing suitable coordinates and frames. In Section 3, we recall the local  well-posedness and high order energy estimates. In Section 4, we prove the global regularity and long time dynamics. In Section 5, we prove the existence of wave operators. In  Section 7, we study the local Cauchy problem for arbitrary large data in $H^{\sigma}_{E}$.

\noindent{\bf Notations.}
We denote  $A\lesssim B$ if there exists some universal constant $C>0$ such that $A\le CB$.
The notation $C_{\alpha_1,...,\alpha_j}$ denotes some constant depending on the parameters $\alpha_1,...,\alpha_j$, and it generally may vary  from line to line.
Denote $\langle x\rangle =\sqrt{1+|x|^2}$.
The Fourier transform $f\mapsto \widehat{f}$ is defined by
\begin{align*}
\widehat{f}(\eta)=\int_{\Bbb R^n} f(x)e^{-i\eta\cdot x}dx.
\end{align*}
The usual $L^2$ type Sobolev spaces $H^s$  on $\Bbb R^n$ are defined by
\begin{align*}
\|f\|_{H^{s}}=\|\langle \eta\rangle^{s}\widehat{f}(\eta)\|_{L^2_{\eta}}.
\end{align*}
Define $H^{\infty}=\cap^{\infty}_{s=0}H^s$. And for $k\in\Bbb Z_+, p\in[1,\infty]$, the $L^p$ type Sobolev spaces $W^{k,p}$ on $\Bbb R^n$  are defined by
\begin{align*}
\|f\|_{W^{k,p}}=\sum^{k}_{l=0}\|D^{l}f\|_{L^p}.
\end{align*}

For $p,q\in[1,\infty]$ and ${\Bbb J}\subset \Bbb R$, denote
\begin{align*}
\|f\|_{L^p_tL^q_x({\Bbb J}\times \Bbb R^n)}=(\int_{\Bbb J}(\int_{\Bbb R^n}|f|^qdx)^{\frac{p}{q}}dt)^{\frac{1}{p}}
\end{align*}
with standard modifications if either $p$ or $q$ equals  $\infty$.

\section{Master equation}

\subsection{Invariant transformations of SMCF}

Let $F:\Sigma\to \Bbb R^{n+2}$ be a 2-codimensional immersion and $\mathbf{H}$ be the mean curvature vector. Then
\begin{align*}
\Delta_g F={\mathbf{H}},
\end{align*}
where $\Delta_g$ denotes the Laplacian on $\Sigma$ of the induced metric $g$  given by
\begin{align*}
g_{ij} ={\partial_{x_i}F\cdot\partial_{x_j}F}.
\end{align*}
Therefore, SMCF reduces to
\begin{align}\label{Hbj00}
\partial_tF=J(F)(\Delta_{g} F), \mbox{ }g_{ij} ={\partial_{x_i}F\cdot\partial_{x_j}F}.
\end{align}

\begin{lemma}\label{dd2}
Let $\Sigma$ be an $n$-dimensional  oriented manifold and $F:\Sigma\times \Bbb I\to \Bbb R^{n+2}$ be a solution to SMCF.   Then, given   a differmorphism $\chi:\Sigma\to \Sigma$,  $F(\chi(x),t)$ is still a solution to SMCF.
Moreover,  for any given  transform $B\in SO(n+2)$ and constant vector $b\in \Bbb R^{n+2}$, $BF(x,t)+b$ solves SMCF as well.
And if  $S\in O(n+2)$ with ${\rm det}(S)=-1$, then  $\acute{F}:=SF(x,t)+b$ solves
\begin{align*}
\partial_t\acute{F}=-J(\acute{F}){\bf H}(\acute{F}).
\end{align*}
\end{lemma}
\begin{proof}
The fact that SMCF is invariant under differmorphisms of $\Sigma$ follows from the invariance of ${\mathbf{H}}$ under coordinate transformations of $\Sigma$.
Given a transform $B\in SO(n+2)$ and constant vector $b\in \Bbb R^{n+2}$, denote $\check{F}(x,t)=BF(x,t)+b$. Then  one has
\begin{align*}
&\partial_t\check{F}= B\partial_t F, \mbox{ }\mbox{ }\partial_i\check{F}\cdot \partial_j\check{F}=\partial_{i}F\cdot\partial_j F\\
&\Delta_{\check{g}}\check{F}=B\Delta_{g}F.
\end{align*}
Given $t\in\Bbb I$,  denote $\check{\Sigma}_t=\check{F}(t)(\Sigma)$. Then  the corresponding normal bundle and tangent bundle of $\check{\Sigma}_t$ at $x$ now are $B(N\Sigma_t) $ and $B(T\Sigma_t)$ respectively. Since $B$ preserves  both the orientation  and the metric of $\Bbb R^{n+2}$, we see
\begin{align}\label{67gb}
J(\check{F})B\eta= BJ(F)\eta, \mbox{ }\forall \eta\in N\Sigma_t.
\end{align}
Thus (\ref{Hbj00}) shows $\check{F}$ solves SMCF as well.

It remain to consider $S\in O(n+2)$ with ${\rm det}(S)=-1$. Observe that it suffices to  replace  (\ref{67gb})  by $J(\acute{F})S\eta=- SJ(\acute{F})\eta$, for any $ \eta\in N\Sigma_t$.

\end{proof}

\subsection{Reduction to graph solutions}

We first do an orthogonal  transformation to put vectors  $\xi_1,\xi_2$ in (\ref{dVBn}) and (\ref{2dVBn}) to be zero.
In fact, there exists an orthogonal  transformation $\mathbb{B}\in O(n+2)$ such that the $(n+1)$-th and $(n+2)$-th components of $ {\Bbb B} E(x)$ are constants, i.e.
\begin{align*}
&({\Bbb B} E(x))^{n+1}=\tilde{a}_1, \mbox{ }({\Bbb B} E(x))^{n+2}=\tilde{a}_2,\mbox{ }\forall x\in \Bbb R^n\\
&\left(({\Bbb B} E(x))^{1},...,({\Bbb B} E(x))^{n}\right)^t={\Bbb A}x+b, \mbox{ }\forall x\in \Bbb R^n,
\end{align*}
for some $\tilde{a}_1,\tilde{a}_2\in \Bbb R$, ${\Bbb A}\in GL(n)$, and $b\in \Bbb R^n$.
Then  (\ref{dVBn}) now reads as
\begin{align*}
\| {\Bbb B}F_0-({\Bbb A}x+b,\tilde{a}_1,\tilde{a}_2)\|_{H^{k}(\Bbb R^n)}\le \epsilon_*.
\end{align*}

Set
\begin{align}\label{78vv}
\mathcal{{{F}}}_0(x)={\Bbb B}F_0({\Bbb A}^{-1}x)-(b,\tilde{a}_1,\tilde{a}_2).
\end{align}
Thus
\begin{align}\label{Gh78vv}
\|  \mathcal{{F}}_0(x)-( x,0,0)\|_{H^{k}(\Bbb R^n)}\lesssim \epsilon_*.
\end{align}

Second, we choose suitable coordinates for $\Sigma$.
Implicit function theorem and Sobolev embedding give
\begin{lemma}\label{bhhjkmn}
Let $n\ge 1, e\in\Bbb N, e>\frac{n}{2}+1$. Assume that
$F_0$   satisfies
\begin{align*}
\| F_0(x)-(x_1,...,x_n,0,0)\|_{H^e_x}\le \varepsilon.
\end{align*}
If $\varepsilon>0$ is sufficiently small depending on $n$, then there exists a homeomorphism  $\Phi:\Bbb R^n\to \Bbb R^n$ such that
\begin{align*}
 (F^1_0(\Phi(x)),...,F^n_0(\Phi(x)))=(x_1,...,x_n),\mbox{ }\forall x\in\Bbb R^{n},
\end{align*}
and $\Phi$ is near the identity map in the sense that
\begin{align}\label{jnnmk878}
\|\Phi-I\|_{C^0}+\|\Phi-I\|_{{\dot{H}^1}\cap {\dot H}^{e}}\lesssim  \varepsilon.
\end{align}
\end{lemma}
\begin{proof}
Let $\Bbb F_0:\Bbb R^{n}\to \Bbb R^n$ be
\begin{align*}
{\Bbb F}_0(x)=(F^1_0(x),...,F^n_0(x)).
\end{align*}
By Sobolev embedding
\begin{align*}
\| \mathbb{F}_0-I\|_{C^1_x(\Bbb R^n)}\lesssim \varepsilon.
\end{align*}
By implicit function theorem, for $\varepsilon$ sufficiently small, there exists a unique $C^1$ map $\Phi:\Bbb R^n\to \Bbb R^n$ such that
\begin{align*}
 \mathbb{F}_0(\Phi(x))=x,\mbox{ }\forall x\in\Bbb R^{n}.
\end{align*}
Formally one has
\begin{align}\label{xcvnmmm}
D^{l}\Phi = (D\mathbb{F}_0)^{-1}(D^{l-1}I-\sum_{j+i=l+1, 1\le i\le l-1}c_{j,i}D^{j}\mathbb{F}_0D^{i}\Phi).
\end{align}
So by induction,
the inequality (\ref{jnnmk878}) follows by (\ref{xcvnmmm}) and the fact that $H^{e}$ is an algebra.
\end{proof}

Let $\mathcal{F}_0$ satisfy (\ref{Gh78vv}).
Set
\begin{align}\label{vVvghbm}
\tilde{x}_1=\mathcal{F}^1_0(x),..., \tilde{x}_n=\mathcal{F}^n_0(x).
\end{align}
Define the map $\widetilde{F}_0:\Bbb R^{n}\to \Bbb R^{n+2}$ by
\begin{align}\label{jnmmml}
\widetilde{F}_0(\tilde{x})=(\tilde{x}_1,...,\tilde{x}_n, \tilde{F}^{n+1}_0(\tilde{x}),\tilde{F}^{n+2}_0(\tilde{x})),
\end{align}
where $\{\tilde{F}^{j}_0(\tilde{x})\}^{n+2}_{j=n+1}$ are defined by
\begin{align}\label{2Vbnnm}
 \tilde{F}^{j}_0(\tilde{x})=\mathcal{F}^{j}_0(\Phi(\tilde{x})).
\end{align}
Thus, because of the transformation invariance in Lemma \ref{dd2}, it suffices to study graph like initial data as (\ref{jnmmml}). To be precise, let us state a parallel result.

{\bf Theorem 1.1'}
{\it Let $n\ge 3$ and $k$ be an integer such that  $k\ge n+4$. There exists a sufficiently small constant $\epsilon>0$ so that if $F_0\in H^{k}_{E}$ satisfies
\begin{align*}
&F_0(x)=(x_1,...,x_n,u^1_0(x),u^2_0(x)),\mbox{ }\forall x\in\Bbb R^n\\
&\|u^1_0\|_{ H^{k}_x}+\|u^2_0\|_{ H^{k}_x} \le \epsilon,
\end{align*}
then there exists a global unique solution to SMCF of the form
$$F(x,t)=(x_1,...,x_n,u_1(x,t),u_2(x,t))$$
with initial data $F_0$ so that $u_1,u_2\in C([-T,T];H^{k})$ for all $T>0$. Moreover, there exist time independent complex valued  functions $\phi_{\pm}\in H^{k_n}$ such that
\begin{align}\label{kmnll}
\lim\limits_{t\to \pm \infty}\|  u^{1}(t,x) + \mathrm{i}u^{2}(t,x) -e^{\mathrm{i}t\Delta}\phi_{\pm} \|_{H^{k_n}_x}=0,
\end{align}
 where $k_n$ is the smallest integer that is bigger than or equal to  $\frac{n}{2}+1$. In addition, if $F_0\in H^{\infty}_{E}$, then $F(t)$ is smooth and belongs to $C([-T,T];H^{\sigma}_{E})$ for all $T>0,\sigma\ge 0$.}

\begin{lemma}\label{vcfghfghb}
Theorem 1.1' implies the global regularity result and asymptotic behavior result (\ref{scatter}) in Theorem 1.1.
\end{lemma}
\begin{proof}
Given $F_0$ satisfying conditions of Theorem 1.1, we have seen there exist ${\Bbb B}\in O(n+2)$, ${\Bbb A}\in GL(n)$, $\tilde{a}_1,\tilde{a}_2\in \Bbb R$, $b\in \Bbb R^{n}$ such that (\ref{Gh78vv}) holds   for the $\mathcal{F}_0$ defined in (\ref{78vv}).
Without loss of generality, we assume ${\rm det}(\Bbb B)=1$, otherwise, it suffices to redefine the complex structure $J$. Furthermore, Lemma \ref{bhhjkmn} shows there exists
a homeomorphism $\Phi:\Bbb R^n\to \Bbb R^n$ such that claims in Lemma \ref{bhhjkmn} hold.
Define
\begin{align}\label{ccVbnnm}
u^1_0(\tilde{x}):=\mathcal{F}^{n+1}_0(\Phi(\tilde{x})), \mbox{ }u^2_0(\tilde{x}):=\mathcal{F}^{n+2}_0(\Phi(\tilde{x})).
\end{align}
Then Lemma \ref{bhhjkmn}  implies $u^1_0,u^2_0$ satisfy
\begin{align*}
\|u^1_0\|_{ H^{k} }+\|u^2_0\|_{ H^{k} } \lesssim \epsilon.
\end{align*}
Thus  Theorem 1.1' shows there exists a global unique  solution to SMCF with initial data $(\tilde{x}_1,...,\tilde{x}_n,u^1_0,u^2_0)$ in $C([-T,T];H^{k}_{E})$ for all $T>0$. Moreover, there exist time independent complex valued  functions $\phi_{\pm}\in H^{k_n}$ such that (\ref{kmnll}) holds. Denote this solution by $\mathbf{F}(\tilde{x},t)$.
Define $F(x,t)$ by
\begin{align}\label{ccVbnnbbm}
F(x,t)={\Bbb B}^{-1}\mathbf{F}( \Phi^{-1}{{\Bbb A}x},t)+{\Bbb B}^{-1}\vec{e}, \mbox{ } \forall (x,t)\in\Bbb R^n\times\Bbb R,
\end{align}
where ${\vec e}=(b,\tilde{a}_1,\tilde{a}_2)$ in (2.3). Then $F$ solves the  SMCF equation with initial data  $F_0$.
Moreover, (\ref{scatter}) follows by (\ref{kmnll}). In addition, if $F_0\in H^{\infty}_{E}$, then $\Phi$ in Lemma \ref{bhhjkmn} is a differmorphism. And thus $F(x,t)$ is smooth and belongs to $C([-T,T];H^{\sigma}_{E})$ for all $T>0,\sigma>0$ by Theorem 1.1'.
\end{proof}

In the rest of Section 2, Section 3, and Section 4.1 to Section 4.2, we prove Theorem 1.1'. If Theorem 1.1' is done, the first two claims of  Theorem 1.1 follow   as a corollary of  Lemma  \ref{vcfghfghb}. In the Section 5, we prove the existence of wave operators.

\subsection{Reduction to quasilinear Schr\"odinger  equations }

Let us consider the case when $F$ is represented by a graph, i.e. $F(x,t)=(x_1,...,x_n,u_1,u_2)$, where $u_1,u_2$ are functions of $x,t$.
Then  our previous work \cite{Li} implies   (\ref{1}) indeed reduces to
\begin{align}\label{sGGmGGkGG3}
\left\{
  \begin{array}{ll}
    \partial_t u_1&=  -\frac{1}{\Lambda  \sqrt{1+|\partial_x u_1|^2}}g^{ij} \partial^2_{x_ix_j}u_2  \hbox{ } \\
    \partial_t u_2&=   \frac{1}{\Lambda  \sqrt{1+|\partial_x u_1|^2}}g^{ij}\partial^2_{x_ix_j} u_1, \hbox{ }
  \end{array}
\right.
\end{align}
where we denote
\begin{align}
\Lambda&=\left(|\partial_x u_2-\frac{\partial_{x}u_2\cdot\partial_x u_1}{1+|\partial_x u_1|^2}\partial_x u_1|^2+|\frac{\partial_x u_1\cdot \partial_x u_2}{1+|\partial_x u_1|^2}|^2+1\right)^{\frac{1}{2}}.
\end{align}
Here, $\partial_xf=(\partial_{x_1}f,...,\partial_{x_n}f)$, and $|\partial_x f|$  denotes its Euclidean norm in $\Bbb R^{n}$.
Letting $\phi=u_1+ \mathrm{{i}}u_2$, one observes that  (\ref{sGGmGGkGG3}) is indeed a quasilinear Schr\"odinger equation:
\begin{align}\label{mGGsGGs}
 \mathrm{i}\partial_t \phi+ \Delta \phi=G(\partial^2_x\phi^{\pm},\partial_x {\phi}^{\pm}),
\end{align}
 where  ${\phi^{+}}$ and ${\phi^{-}}$ denote $\phi$ and $\bar{\phi}$ respectively.
The RHS of (\ref{mGGsGGs}) can be written as
\begin{align}\label{dGGdGGmGGsGGs}
{\sum \alpha_{iji'j'}} \partial^2_{x_ix_j}\phi^{\pm}  \partial_{x_{i'}}\phi^{\pm} \partial_{x_{j'}}\phi^{\pm}+\mathcal{R},
\end{align}
where $\alpha_{iji'j'}$ are universal constants   and the remainder  {$\mathcal{R}$} point-wisely fulfills
\begin{align}\label{mGGq}
|\partial^{j}_x \mathcal{R}|\le C_{j}\sum_{i_1+...+i_m\le j,4\le m\le 2^{10j} }|\partial^{i_1}_x\partial^2_{x}\phi^{\pm}||\partial^{i_2}_x\partial_x\phi^{\pm}|...|\partial^{i_m}_x\partial_x\phi^{\pm}|
\end{align}
for any $j\ge 0$ if $\|\partial_x\phi^{\pm}\|_{L^{\infty}_{t,x}}$ is sufficiently small.

\begin{remark}
The orthogonal transformation ${\Bbb B}$ in (\ref{78vv}) is important. Observe that if one directly chooses a diffeomorphism $\widetilde{\chi}:\Bbb R^n\to \Bbb R^n$ so that $F^{j}_0(\widetilde{\chi}(x))=x_j$ for all $j=1,...,n$, the appropriate unknown function shall be
$$\widetilde{ \phi}(x,t)=(F^{n+1}(\widetilde\chi(x),t)-\xi_1\cdot\widetilde\chi(x)-a_1)+  \mathrm{i}(F^{n+2}(\widetilde\chi(x),t)-\xi_2\cdot\widetilde\chi(x)-a_2).
$$
Then $\widetilde{ \phi}$ satisfies a quasilinear Schr\"odinger equation
\begin{align*}
 \mathrm{i}\partial_t \widetilde{\phi}+\sum A_{ij}(x)\partial^2_{ij}\widetilde{ \phi}+ \sum b_{j}(x)\partial_{j}\widetilde{ \phi}=\widetilde{G}(\partial^2_x\widetilde{\phi}^{\pm},\partial_x{\widetilde{\phi}^{\pm}},x),
\end{align*}
whose linear part is a Schr\"odinger operator with variable coefficients and the nonlinear part is
is quadratic in $\widetilde{\phi}$. But observe that the linear part of (\ref{mGGsGGs}) is purely the free Schr\"odinger operator and the nonlinear part is cubic in ${\phi}$.
Comparing these two different master equations, we see the orthogonal transformation ${\Bbb B}$ in (\ref{78vv}) essentially simplifies the whole issue. This also reveals the importance of gauge invariance.
\end{remark}

\section{Local Cauchy theory and energy estimates}

Let $F :{\Bbb I}\times \Sigma\to  \Bbb R^{n+2}$ be an immersion.  For each $t\in \Bbb I$, let $g=g(t)$
denote  the induced metric on $\Sigma$, and  denote  $d\mu=d\mu(t)$ the corresponding volume form on $\Sigma$. The pullback bundle $F^*T \Bbb R^{n+2}$ defined over the base manifold $\Bbb I\times \Sigma$ naturally splits into  the ``spatial'' subbundle and the normal bundle. Pulling back the metric and connection of $\Bbb R^{n+2}$ naturally  induces a metric $g^{\mathfrak{{N}}}$ and connection $\nabla^{\mathfrak{{N}}}$ defined on the normal bundle. For simplicity, we write $\nabla$  instead of  $\nabla^{\mathfrak{{N}}}$. See Section 7.2  for a detailed presentation.

Assume that $F:\Bbb I\times \Sigma\to \Bbb R^{n+2}$ is an immersion given by
\begin{align}\label{hhhhhhhhhh}
F(t,x_1,...,x_n)=(x_1,...,x_n,u_1(t,x_1,..,x_n),u_{2}(t,x_1,...,x_n)).
\end{align}
Denote $\mathbf{A}$   the associated second fundamental of $\Sigma_t:=Graph(u)$.
Define the Sobolev norm of ${\bf A}$ by
\begin{align*}
\|{\bf{ A}}\|_{H^{m,p}}=\left(\sum^{m}_{i=0}\int_{\Sigma} |\nabla^{i} {\bf A}|^p_{g}d\mu\right)^{\frac{1}{p}}.
\end{align*}
Let $D^2$ denote the common Hessian operator in $\Bbb R^n$.
And define the usual Sobolev norm of $D^2 u$ by
\begin{align*}
\|D^2 u\|_{W^{m,p}}=\left(\sum^{m}_{j=0}\int_{\Sigma} |D^{j} D^2u|^pdx\right)^{\frac{1}{p}}.
\end{align*}

We recall the   energy estimates obtained by Song-Sun \cite{SGGoGGnGGgGG1} and Song \cite{SGGoGGnGGgGG2}   in the following  lemma.
\begin{lemma}[\cite{SGGoGGnGGgGG1,SGGoGGnGGgGG2}]\label{GeG}
Let $F:\Bbb I\times \Bbb R^n\to \Bbb R^{n+2}$ be of the form (\ref{hhhhhhhhhh}). With above  notations, we have
\begin{align}
|{\bf A}|^2_g&\le|D^2 u|^2\le (1+|Du|^2)^3|{\bf A}|^2_g, \mbox{ }\forall (t,x)\in\Bbb I\times \Sigma\label{xv32}.
\end{align}
Moreover,  given integer $j\ge 0$ there exists a polynomial $P_{j}$ depending only on $j$ such that
\begin{align}
|\nabla^{j}{\bf A}|_g&\le|D^{j+2} u|+P_{j}(|Du|)\sum|D^{i_1+1}u|...|D^{i_m+1}u|\label{xv33}\\
|D^{j+2}u|&\le (1+|Du|^2)^{\frac{j+3}{2}}|\nabla^2{\bf A}|_g  +P_{j}(|Du|)\sum|D^{i_1+1}u|...|D^{i_m+1}u|\label{xv34}
\end{align}
where the summations are taken over the indices $(i_1,...,i_m)$ fulfilling
\begin{align*}
i_1+...+i_m=j+1, \mbox{ }i_1\ge i_2\ge ...\ge i_m, \mbox{ } j \ge i_1\ge i_m\ge 1.
\end{align*}
Assume that $|Du|\le \gamma$, then for any $j\ge 0$,
\begin{align}
\|{\bf A}\|_{H^{j,2}}\le C_{\gamma} \sum^{j+1}_{m=1}\|D^2u\|^{m}_{W^{j,2}}\label{xv35}.
\end{align}
If $\|Du\|_{L^{\infty}}\le \gamma_1$,  $\|D^2u\|_{L^{\infty}}\le \gamma_2$, then for any $j\ge 0$,
\begin{align}\label{xv36}
\|D^2u\|_{W^{j,2}}\le C_{\gamma_1,\gamma_2}( \sum^{j}_{i=1} \|{\bf A}\|^{i}_{H^{j,2}}+\sum^{j}_{i=2} \|D^2u\|^{i}_{W^{j,2}}).
\end{align}
Let $F(t,x)$ be a  smooth solution of SMCF. Then one has
\begin{align}\label{xv37}
\frac{d}{dt}\int_{\Sigma }|\nabla^{l} {\bf A} |^2_{g}d\mu\le C \max_{\Sigma }|{\bf A}|^2_{g} \int_{\Sigma}|\nabla^{l} {\bf A} |^2_{g}d\mu.
\end{align}
\end{lemma}

Our previous work \cite{Li} has studied local Cauchy problem for graph solutions. We restate the results as follows for reader's convenience.
\begin{lemma}\label{1xx1xxHxxjxxK}
Given $n\ge 2$, let $s\in \Bbb R$, $s>\frac{n+5}{2}$. There exists a small  constant $\varepsilon>0$   depending only on $n,s$ such that if
      $$\|u^1_0\|_{H^{s}_x}+\|u^2_0\|_{H^{s}_x} \le \varepsilon,$$
then there exists a unique  local solution $F$ to SMCF  of the form
\begin{align}\label{Fxxgxxm}
F(t,x)=(x_1,...,x_n,u_1(t,x),u_2(t,x)),\mbox{  }(u_1,u_2)\upharpoonright_{t=0}=(u^0_1,u^0_2),
\end{align}
such that $u_1,u_2\in C([-1,1];H^{s})$. Moreover, the continuous dependence of solutions on initial data in $H^s$  and the regularity persistence hold as well.
\end{lemma}

The following lemma states a conditional global regularity.
\begin{lemma}\label{KK}
Let $n\ge 2$. Assume that $ u^1_0(x),u^2_0(x)\in H^{\infty}$. Then the solution of SMCF of the form (\ref{Fxxgxxm}) with initial data $u^1_0,u^2_0$ belongs to $C([0,T];H^{\sigma}_{E})$ for all $\sigma\ge 0$ as long as
\begin{align*}
\sup_{t\in [0,T)}\|{\bf A}(t)\|_{H^{0,\infty}}<\infty.
\end{align*}
\end{lemma}
\begin{proof}
By  Lemma 3.9 of \cite{Li},  there exists a $T'>0$ depending only on $\|u_0\|_{H^{[\frac{n}{2}]+3}}$ such that $u_0$ evolves to a unique solution of SMCF fulfilling   $u\in C([0,T'];H^{\infty})$ . Moreover, Corollary 3.8 of \cite{Li} implies that  if $u$ fails to belong to  $C([0,T_*];H^{\infty})$ for some $T_*>0$ then $\lim_{\tau\to T_*}\sup_{t\in [0,\tau)}\|{\bf A}(t)\|_{H^{0,\infty}}=\infty$. Combining these two facts proves  our lemma.
\end{proof}

 \section{Proof of main theorem in $n\ge 3$}

The following is the Strichartz estimates for linear Schr\"odinger equations.

\begin{lemma}(c.f. \cite{MT})
Let $n\ge 1$. We call $(a,b)$ an $L^2$ admissible pair if
\begin{align}
 \frac{2}{a}+\frac{n}{b}=\frac{n}{2}, \mbox{ }2\le a,b\le \infty, \mbox{ }(a,b,n) \neq (2,\infty,2).
\end{align}
Assume that $v$ solves
\begin{align}
\left\{
  \begin{array}{ll}
    (i\partial_t +\Delta)v=G, & \hbox{ } \\
    v(0,x)=v_0(x). & \hbox{ }
  \end{array}
\right.
\end{align}
Let $(a,b)$ and $(m,z)$ be any $L^2$ admissible pairs, and denote $m',z'$ the conjugate of $m$ and $z$ respectively. Then there exists  a constant $C>0$  depending only on $a,b,n,m,z$ such that
\begin{align*}
\|v\|_{L^{a}_tL^b_x(\Bbb J\times \Bbb R^n)}\le C( \|v_0\|_{L^2_x}+\|G\|_{L^{m'}_tL^{z'}_x(\Bbb J\times\Bbb R^n)})
\end{align*}
for any interval $\Bbb J$ containing $0$.
\end{lemma}

The following lemma provides  some elementary arithmetic facts  we need below.
\begin{lemma}\label{poikkm}
Let $n\in\Bbb Z_+$, $n\ge 2$.  If $n$ is even, then
\begin{align}
|\partial^{\frac{n}{2}+1}_x \mathcal{R}|\le C_{n}\sum_{i_1+...+i_m\le \frac{n}{2}+1,4\le m\le 2^{10n} }|\partial^{i_1}_x\partial^2_{x}\phi^{\pm}||\partial^{i_2}_x\partial_x\phi^{\pm}|...|\partial^{i_m}_x\partial_x\phi^{\pm}|\label{Jnkm}\\
|\partial^{\frac{n}{2}+1}_x (\partial^2_x\phi^\pm\partial_x \phi^\pm\partial_x \phi^\pm)|\le C_{n}\sum_{i'_1+...+i'_m\le \frac{n}{2}+1,3\le m\le 4n }|\partial^{i'_1}_x\partial^2_{x}\phi^{\pm}||\partial^{i'_2}_x\partial_x\phi^{\pm}|...|\partial^{i'_m}_x\partial_x\phi^{\pm}|.\label{2Jnkm}
\end{align}
For each term in the sum of  (\ref{Jnkm}) there exist  at least two $i_l$ in $\{i_l\}^{m}_{l=1}$  such that $1\le i_{l}\le \frac{n}{4}+2$. And  for each term in the sum of  (\ref{2Jnkm}) there exist at least two   $i'_l$ in $\{i'_l\}^{m}_{l=1}$  such that $1\le i'_{l}\le \frac{n}{4}+2$.

If $n$ is odd, then
\begin{align}
|\partial^{\frac{n+1}{2}+1}_x \mathcal{R}|\le C_{n}\sum_{i_1+...+i_m\le \frac{n+1}{2}+1,4\le m\le 2^{10n} }|\partial^{i_1}_x\partial^2_{x}\phi^{\pm}||\partial^{i_2}_x\partial_x\phi^{\pm}|...|\partial^{i_m}_x\partial_x\phi^{\pm}|\label{hJnkm}\\
|\partial^{\frac{n+1}{2}+1}_x (\partial^2_x\phi^\pm\partial_x \phi^\pm\partial_x \phi^\pm)|\le C_{n}\sum_{i'_1+...+i'_m\le \frac{n+1}{2}+1,3\le m\le 4n }|\partial^{i'_1}_x\partial^2_{x}\phi^{\pm}||\partial^{i'_2}_x\partial_x\phi^{\pm}|...|\partial^{i'_m}_x\partial_x\phi^{\pm}|.\label{h2Jnkm}
\end{align}
For each term in the sum of  (\ref{hJnkm}) there exist  at least two $i_l$ in $\{i_l\}^{m}_{l=1}$ so that $1\le i_l\le \frac{n+1}{4}+2$. And  for each term in the sum of  (\ref{h2Jnkm}) there exist at least  two $i'_l$ in $\{i'_l\}^{m}_{l=1}$  such that $1\le i'_{l}\le \frac{n+1}{4}+2$.
\end{lemma}
\begin{proof}
This follows by the Leibnitz rule  and direct calculations.
\end{proof}

\subsection{Proof of global regularity}
By Duhamel principle,
the solution of (\ref{mGGsGGs}) can be expressed by
\begin{align*}
 \phi(t)=e^{\mathrm{i}t\Delta}\phi_0-i\int^{t}_{0}e^{\mathrm{i}(t-s)\Delta} G(\partial^2_x\phi^{\pm},\partial_x\phi^{\pm})(s)ds.
\end{align*}
Then Strichartz estimates give
\begin{align}\label{1ma}
 \|\phi(t)\|_{L^{a}_tW^{l,b}_x}\lesssim  \|\phi_0\|_{H^{l}_x}+
 \|G(\partial^2_x\phi^{\pm},\partial_x\phi^{\pm})\|_{L^{2}_tW^{l,\frac{2n}{n+2}}_x}
\end{align}
for any $L^2$ admissible pair  $(a,b)$ and any $l\ge 0$.

\subsubsection{ $n\ge 4$ is even.}
Let's consider the case when $n\ge 4$ and is even.

{\it Step 1. Setup of bootstrap.}
Given $T\ge0$, define
\begin{align*}
 \|\phi\|_{ {X}_T}=&\|\phi\|_{L^{\infty}_tW^{\frac{n}{2}+1,2}_x ([-T,T]\times\Bbb R^n)}+\|\phi\|_{L^{2}_tW^{\frac{n}{2}+1 ,\frac{2n}{n-2}}_x ([-T,T]\times\Bbb R^n)}\\
 \|\phi\|_{ \mathcal{{X}}_T}=&\|\phi\|_{ {X}_T}+\| \phi \|_{L^{\infty}_tH^{k}_x([-T,T]\times\Bbb R^n)}.
\end{align*}
Let $\mathcal{T}$ be the maximal time such that
\begin{align}\label{1xxmxxaxxbxx1}
\sup_{T\in[0,\mathcal{T}]} \| \phi \|_{\mathcal{X}_{T}} \le C_0\epsilon.
\end{align}
By local theorem in Lemma \ref{1xx1xxHxxjxxK} and Sobolev embedding, one has  $\mathcal{T}>0$ for some $C_0>0$.

{\it Step  2. Leading nonlinear term.}
Letting $l=\frac{n}{2}+1$ in (\ref{1ma}),
from  Lemma \ref{poikkm} we obtain by H\"older inequality and Sobolev embeddings that
\begin{align}
&\| {\partial^2_x\phi }|\partial_x\phi|^2 \|_{L^{2}_tW^{l,\frac{2n}{n+2}}_x}\nonumber\\
&\lesssim\sum_{3\le m\le 4l} \sum_{1\le  j_3,...,j_m\le \frac{n}{2}+3}
\sum_{1\le j_1,j_2\le \frac{n}{4}+2}\|\partial^{j_1}_x\phi\|_{L^{2}_tL^{4}_x}\|\partial^{j_2}_x \phi\|_{L^{\infty}_tL^{\frac{4n}{n+4}}_x}\|\cdot\|_{L^{\infty}_{t,x}}....\|\partial^{j_m}_x \phi\|_{L^{\infty}_{t,x}}\nonumber\\
&\lesssim \sum_{3\le m\le 4l} \sum_{1\le  j_3,...,j_m\le \frac{n}{2}+3}
\sum_{1\le j_1,j_2\le \frac{n}{4}+2}\|\phi\|_{L^{2}_tW^{j_1+\frac{n}{4}-1 ,\frac{2n}{n-2}}_x}\| \phi\|_{L^{\infty}_tW^{j_2+\frac{n}{4}-1,2}_x}\|\cdot\|_{L^{\infty}_{t,x}}... \|\partial^{j_{m}}_x\phi\|_{L^{\infty}_{t,x}}.\label{mmvbMmm}
\end{align}
We have several remarks for (\ref{mmvbMmm}):
Since $j_3,...,j_m\le \frac{n}{2}+3$, $j_1,j_2\le \frac{n}{4}+2$, there holds
\begin{align*}
\max_{i=3,...,m}|j_i+ \frac{n}{4}-1|\le n+2,\mbox{ }\max_{i=1,2}|j_i+\frac{n}{4}-1 |\le \frac{n}{2}+1,
\end{align*}
which enables us to  apply $k\ge n+4$ to dominate $\| \phi\|_{L^{\infty}_tW^{j_2+\frac{n}{4}-1,2}_x}, \|\partial^{j_{i}}_x\phi\|_{L^{\infty}_{t,x}}, i=3,...,m$, by $\|\phi\|_{L^{\infty}_tH^{k}_x}$, and to absorb $\|\phi\|_{L^{2}_tW^{j_1+\frac{n}{4}-1 ,\frac{2n}{n-2}}_x}$ into $\|\phi\|_{\mathcal{X}_{T}}$.
Thus we have proved
\begin{align*}
\|| {\partial^2_x\phi }||\partial_x\phi|^2  \|_{L^{2}_tW^{\frac{n}{2}+1,\frac{2n}{n+2}}_x([-T,T]\times\Bbb R^n)}\lesssim  \|\phi\|^{3}_{\mathcal{X}_{T}}.
\end{align*}

{\it Step 3. Remainder nonlinear terms.} Similarly for $l=\frac{n}{2}+1$ one obtains from  Lemma \ref{poikkm} that
\begin{align*}
&\| \mathcal{R}  \|_{L^{2}_tW^{l,\frac{2n}{n+2}}_x}\nonumber\\
&\lesssim\sum_{4\le m\le 2^{10l}}\sum_{1\le i_1,i_2\le \frac{n}{4}+2} \sum_{1\le i_3,...,i_{m}\le \frac{n}{2}+3}\|\partial^{i_1}_x \phi\|_{L^{\infty}_tL^{\frac{4n}{n+4}}_x}\|\partial^{i_2}_x\phi\|_{L^{2}_tL^{4}_x}\|\partial^{i_3}_x\phi\|_{L^{\infty}_{t,x}}...\|\partial^{i_m}_x\phi\|_{L^{\infty}_{t,x}}
\nonumber\\
&\lesssim\sum_{4\le m\le 2^{10l}}\sum_{1\le i_1,i_2\le \frac{n}{4}+2} \sum_{1\le i_3,...,i_{m}\le \frac{n}{2}+3}\| \phi\|_{L^{\infty}_tW^{i_1+\frac{n}{4}-1,2}_x}\|\phi\|_{L^{2}_tW^{i_2+\frac{n}{4}-1 ,\frac{2n}{n-2}}_x}\|\phi\|^{m-2}_{L^{\infty}_tH^{k}_x}.
\end{align*}
Therefore,
 \begin{align*}
\|\mathcal{R} \|_{L^{2}_tW^{\frac{n}{2}+1,\frac{2n}{n+2}}_x([-T,T]\times\Bbb R^n)}\lesssim  \sum^{C_{n}}_{j\ge 4}\|\phi\|^{j}_{\mathcal{X}_T}.
\end{align*}

{\it Step 4. Close bootstrap of $X_T$ norm.}
Now, by (\ref{1ma}) we summarize   Step 2 and Step 3 as
\begin{align*}
\|\phi\|_{{X}_T}\lesssim \|\phi_0\|_{H^k_x}+\sum^{C_n}_{l=3}\|\phi \|^l_{\mathcal{{X}}_T}.
\end{align*}
Thus for $T\in[0,\mathcal{T}]$ one has
\begin{align}\label{hbnmm}
\|\phi\|_{{X}_T}\le C_n \|\phi_0\|_{H^k_x}+C^3_n(C_0\epsilon)^3.
\end{align}

 {\it Step 5. Close bootstrap of $H^k$ norm.}

By Sobolev  inequality and the definition of $X_T$ norm, we see
 \begin{align}\label{VbbB}
\| \phi(t)\|_{L^2_tW^{2,\infty}_x([-T,T]\times\Bbb R^n)}\lesssim \|\phi\|_{L^2_tW^{\frac{n}{2}+1,\frac{2n}{n-2}}_x([-T,T]\times\Bbb R^n)}\lesssim  \|\phi\|_{X_T}.
\end{align}
Hence, for any $t\in[0,\mathcal{T}]$  we obtain
\begin{align}\label{jknml090}
\int^{t}_0\| \phi(t)\|^2_{W^{2,\infty}}ds\lesssim  \epsilon^2.
\end{align}
By (\ref{xv32}) and (\ref{jknml090}), for $t\in[0,\mathcal{T}]$, we thus get
\begin{align*}
\int^{t}_0\max_{\Sigma}|{\bf A}|^2_g(s)ds\lesssim \epsilon^2.
\end{align*}
Then for any $t\in[0,\mathcal{T}]$, $j\in \Bbb N$, (\ref{xv37}) and Gronwall inequality yield
\begin{align*}
\int_{\Sigma} |\nabla^{j}{\bf A}|^2_g(t)d\mu\le 2\int_{\Sigma} |\nabla^{j}{\bf A}|^2_g(0)d\mu,
\end{align*}
provided  $0<\epsilon\ll1$.
Therefore,  (\ref{xv35}) implies for any $t\in[0,\mathcal{T}]$, $0\le j\le k$,
\begin{align}\label{MxxxkxxxL}
\int_{\Sigma} |\nabla^{j}{\bf A}|^2_g(t)d\mu\lesssim \epsilon.
\end{align}
Noticing that  the RHS of (\ref{xv36}) with $j=k-2$  is quadratic in $\|D^2u\|_{W^{k-2,2}}$, and  $\|D^2u\|_{W^{k-2,2}}$ is small by bootstrap assumption,  (\ref{MxxxkxxxL}) further yields  for any $t\in [0,\mathcal{T}]$,
\begin{align}\label{HxxxnxxxbxxxgxxxM}
\|D^2u(t)\|_{W^{k-2,2}}\le C_1 \|\phi_0\|_{H^{k}}
\end{align}
for some $C_1>0$ if $0<\epsilon\ll1$.  (\ref{HxxxnxxxbxxxgxxxM}) has given admissible  bounds of $\|\phi\|_{{\dot H}^{j}}$ with $2\le j\le k$.  By the definition of $X_{T}$ and (\ref{hbnmm}),  we see
\begin{align}\label{tcccmcccamcccm}
\| \phi(t)\|_{H^2_x}\le \|\phi \|_{X_{T}}\le C_n (C_0\epsilon)^2+C_n\|\phi_0\|_{H^k}.
\end{align}
Hence, (\ref{HxxxnxxxbxxxgxxxM}) and  (\ref{tcccmcccamcccm}) yield
\begin{align*}
\|\phi\|_{L^{\infty}_t H^k_x}\le C_n (C_0\epsilon)^2+C_n\|\phi_0\|_{H^k}.
\end{align*}

 {\it Step 6. Close bootstrap.}
As a summary of Step 4 and Step 5, we have obtained
\begin{align}\label{HxnxbxM}
\|u(t)\|_{\mathcal{X}_T}\le C_n (C_0\epsilon)^2+C_n\|\phi_0\|_{H^k},
\end{align}
for some $C_n>0$ and any $T\in [0,\mathcal{T}]$. (The inverse direction $t\in[-\mathcal{T},0]$ follows by a time reflection and defining $J(F)$ as the opposite direction rotation.)
Set $C_0$ to fulfill
\begin{align*}
C_0>10+4C_n,
\end{align*}
then choose $\epsilon>0$ to be sufficiently small such that
\begin{align*}
C_nC^2_0\epsilon\le \frac{1}{100}.
\end{align*}
By (\ref{HxnxbxM}) and bootstrap, we have proved (\ref{1xxmxxaxxbxx1}) holds with $C_0\epsilon$ replaced by $\frac{1}{2}C_0\epsilon$.

\subsubsection{ $n\ge 3$ is odd.}
Let's consider the case when $n\ge 3$ and is odd.

{\it Step 1'. Setup of bootstrap.}
Given $T\ge0$, define
\begin{align*}
 \|\phi\|_{ {X}_T}&=  \|\phi\|_{L^{\infty}_tW^{\frac{n+1}{2}+1,2}_x([-T,T]\times\Bbb R^n)}+\|\phi\|_{L^{2}_tW^{\frac{n+1}{2}+1 ,\frac{2n}{n-2}}_x([-T,T]\times\Bbb R^n)}\\
\|\phi\|_{ \mathcal{{X}}_T}&=\|\phi\|_{ {X}_T}+\| \phi \|_{L^{\infty}_tH^{k}_x([-T,T]\times\Bbb R^n)}.
\end{align*}
Let $\bf{T}$ be the maximal time such that
\begin{align*}
\sup_{T\in[0,\bf{T}]} \| \phi(t)\|_{\mathcal{X}_{T}} \le C_0\epsilon.
\end{align*}

{\it Step 2'. Leading term.}
Take $l=\frac{n+1}{2}+1$ in (\ref{1ma}).
Then from  Lemma \ref{poikkm} we obtain by H\"older inequality and Sobolev embeddings that
\begin{align}
&\| {\partial^2_x\phi }|\partial_x\phi|^2\|_{L^{2}_tW^{l,\frac{2n}{n+2}}_x}\\
&\lesssim \sum_{3\le m\le 4l} \sum_{1\le j_3,...,j_m\le \frac{n+1}{2}+2}\sum_{1\le j_1,j_2 \le \frac{n+1}{4}+2}\|\partial^{j_1}_x\phi\|_{L^{2}_tL^{\frac{8n}{n+1}}_x}\|\partial^{j_2}_x \phi\|_{L^{\infty}_tL^{\frac{4n}{n+3}}_x}\|\partial^{j_3}_x\phi\|_{L^{\infty}_tL^{\frac{8n}{n+1}}_x}\|\cdot\|_{L^{\infty}_{t,x}}...\|\partial^{j_m}_x\phi\|_{L^{\infty}_{t,x}}\nonumber\\
&\lesssim
\sum_{3\le m\le 4l} \sum_{1\le j_1,j_2 \le \frac{n+1}{4}+2}\| \phi\|_{L^{\infty}_tW^{j_1+\frac{n+1}{4}-1,2}_x}\|\phi\|_{L^{2}_tW^{j_2+\frac{n+1}{4}-1 ,\frac{2n}{n-2}}_x} \|\phi\|^{m-2}_{L^{\infty}_tH^{k}_x}.\label{vbMmm}
\end{align}
 Since $j_1,j_2\le \frac{n+1}{4}+2$, there holds
\begin{align*}
\max_{i=1,2}|j_i+ \frac{n+1}{4}-1|\le \frac{n+1}{2}+1,
\end{align*}
which enables us to  dominate $\|\phi\|_{L^{2}_tW^{j_2+\frac{n+1}{4}-1 ,\frac{2n}{n-2}}_x} $ by $\|\phi\|_{L^{2}_tW^{\frac{n+3}{2},\frac{2n}{n-2}}_x}$.
Thus we obtain that
\begin{align*}
&\| {\partial^2_x\phi }|\partial_x\phi|^2 \|_{L^{2}_tW^{\frac{n+1}{2}+1,\frac{2n}{n+2}}_x}\lesssim \|\phi\|^2_{ \mathcal{{X}}_{T}}\|\phi\|_{L^{\infty}_tH^{k}_x}.
\end{align*}

{\it Step 3'. Remainder nonlinear terms.} Similarly we get
 \begin{align*}
\|\mathcal{R} \|_{L^{2}_tW^{\frac{n+1}{2}+1,\frac{2n}{n+2}}_x}\lesssim  \sum^{C_{n}}_{l\ge 4}\|\phi\|^{l}_{\mathcal{X}_{T}}.
\end{align*}

{\it Step 4' to Step 5' Close Bootstrap.} The rest steps are the same as even $n$. Notice that for odd $n\ge 3$ the key estimate  (\ref{VbbB}) shall be replaced by
 \begin{align*}
\| \phi(t)\|_{L^2_tW^{2,\infty}_x([-T,T]\times\Bbb R^n)}\lesssim \|\phi\|_{L^2_tW^{\frac{n+1}{2}+1,\frac{2n}{n-2}}_x([-T,T]\times\Bbb R^n)}\lesssim  \|\phi\|_{X_T}.
\end{align*}
 Thus ${\bf T}=\infty$.

\subsection{Long time dynamics}

(\ref{scatter}) is  indeed a modified  scattering type result.

\begin{lemma}
Let $k_n= \frac{n}{2}+1$ for even $n$ and  $k_n=\frac{n+1}{2}+1 $  for odd $n$.
There exist $\phi_{\pm}\in H^{k_n}$ such that
\begin{align}\label{vvvva1}
\lim_{t\to\pm\infty}\| \phi-e^{ it\Delta}\phi_{\pm}\|_{H^{k_n}_x}=0.
\end{align}
\end{lemma}
\begin{proof}
It is standard to deduce (\ref{vvvva1}) from
\begin{align}\label{cvbnmnbvc}
\| G({\partial^2_x\phi^{\pm} }, \partial_x\phi^{\pm}) \|_{L^{2}_tW^{k_n,\frac{2n}{n+2}}_x} \lesssim  1.
\end{align}
(\ref{cvbnmnbvc}) has been proved in Section 4.1. So our lemma follows.
\end{proof}

\subsection{From graph solution to the original solution}

 Lemma \ref{vcfghfghb}
 has shown under the orthogonal transformation ${\Bbb B}$ and coordinate transform $\Phi$, it suffices to consider graph solutions. Transforming (\ref{vvvva1}) back to the original coordinates gives (\ref{scatter}) in Theorem 1.1. Furthermore, if $F_0\in H^{\infty}_{E}$,  by  Lemma   \ref{KK}, to prove $F(t)\in C([-T,T];H^{\sigma}_{E})$ for all $\sigma\ge 0$ and $T>0$, it suffices to notice that the homoeomorphism $\Phi$ in Lemma 2.2 now is a diffeomorphism and there holds
 \begin{align*}
 \sup_{t\in \Bbb R}\|{\bf A}(t)\|_{H^{0,\infty}}\lesssim C(\|D\phi\|_{L^{\infty}_{t,x}})\|D^2\phi\|_{L^{\infty}_{t,x}}\lesssim \|\phi\|_{L^{\infty}_tH^{k}_x}\lesssim 1.
 \end{align*}

\section{Existence of wave operators}

In this section, we always assume $F$ is a graph solution, i.e. $F(x,t)=(x,u_1(x,t),u_2(x,t))$.
Recall that $(u_1,u_2)$ satisfies (\ref{sGGmGGkGG3}), and
$\phi= u_1 + \mathrm{{i}} u_2$ fulfills  (\ref{mGGsGGs}), i.e.
\begin{align}\label{xxmGGsGGs}
 \mathrm{i}\partial_t \phi+ \Delta \phi= G(\partial^2_x\phi^{\pm},\partial_x{\phi}^{\pm}),
\end{align}
where   $G$ is an explicit function of $\partial^2_x\phi^{+},\partial^2_x\phi^{-},\partial_x\phi^{+},\partial_x\phi^{-}$.

Let $\phi_{\infty}\in H^{k}$ be a given function with $\|\phi_{\infty}\|_{H^{k}}\le \varepsilon_*\ll 1$.
Given $L>0$, consider the equation
\begin{align}\label{xmGGsGGs}
 \phi_{L}(t)=e^{it \Delta } \phi_{\infty}-i\int^{L}_{t}e^{i(t-s)\Delta} G(\partial^2_x\phi^{\pm}_{L}(s),\partial_x{\phi}^{\pm}_{L}(s))ds.
\end{align}
Observe from  (\ref{xmGGsGGs}) that
\begin{align}
 \phi_{L}(L)&=e^{i \Delta L} \phi_{\infty}\\
  \mathrm{i}\partial_t \phi_{L}+ \Delta \phi_{L}&= G(\partial^2_x\phi^{\pm}_{L},\partial_x{\phi}^{\pm}_{L})\label{78jnjll}.
\end{align}
Then define $F_{L}(t):\Bbb R^n\to \Bbb R^{n+2}$ to be
\begin{align*}
 F_{L}(x,t)&=(x,\Re \phi_{L}(x,t),\Im \phi_{L}(x,t)).
\end{align*}
We see $F_{L}(t)$ now solves SMCF. By Section 4, we find the initial data $F_{L}(L)$  which satisfies
\begin{align*}
&F_{L}(L)=(x,\Re \phi_{L}(x,L),\Im \phi_{L}(x,L))= (x,\Re e^{i L \Delta} \phi_{\infty},\Im e^{i L \Delta} \phi_{\infty})\\
&\|F_{L}(L)-(x,0,0)\|_{H^{k}_x}\le \varepsilon_*\ll 1,
\end{align*}
evolves to a global solution $F_{L}(t)$ of SMCF for $t\in \Bbb R$  satisfying
\begin{align}\label{67000}
\|F_{L}(t)-(x,0,0)\|_{L^{\infty}_tH^{k}_x\cap {L^{2}_tW^{k_n,\frac{2n}{n-2}}_x}}+\|\partial_tF_{L}(t)\|_{L^{\infty}_tH^{k-2}_x\cap {L^{2}_tW^{k_n-2,\frac{2n}{n-2}}_x}}\le C\varepsilon_*,
\end{align}
for some universal constant $C>0$ depending only on $n$.
Moreover, applying  arguments of Section 4 to (\ref{xmGGsGGs}), we have the following spacetime bounds:
\begin{align*}
\|\phi\|_{L^{2}_tW^{k_n,\frac{2n}{n-2}}_x([L,\infty)\times \Bbb R^n)} &\le C \|e^{it\Delta}\phi_{\infty}\|_{_{L^{2}_tW^{k_n,\frac{2n}{n-2}}_x([L,\infty)\times \Bbb R^n)}}+C\varepsilon^2_*\|\phi\|_{L^{2}_tW^{k_n,\frac{2n}{n-2}}_x([L,\infty)\times \Bbb R^n )}  \\
\|\phi\|_{L^{2}_tW^{k_n,\frac{2n}{n-2}}_x([M,L)\times \Bbb R^n)} &\le C \|e^{it\Delta}\phi_{\infty}\|_{_{L^{2}_tW^{k_n,\frac{2n}{n-2}}_x([M,L)\times \Bbb R^n)}}+ C\varepsilon^2_*\|\phi\|_{L^{2}_tW^{k_n,\frac{2n}{n-2}}_x([M,L)\times \Bbb R^n )} , \forall M<L,
\end{align*}
 for some universal constant $C>0$ depending only on $n$. Thus for $\varepsilon_*$ sufficiently small  one has
 \begin{align}\label{667000}
\|\phi\|_{L^{2}_tW^{k_n,\frac{2n}{n-2}}_x[M,\infty)\times \Bbb R^n} \le C\|e^{it\Delta}\phi_{\infty}\|_{_{L^{2}_tW^{k_n,\frac{2n}{n-2}}_x[M,\infty)\times \Bbb R^n}}, \mbox{ }\forall M\in \Bbb R.
\end{align}
By  (\ref{67000}) there exists a sequence $L_j\to \infty$ such that $F_{L_j}(t)$ converges weakly in $L^{\infty}_tH^{k}_x\cap {L^{2}_tW^{k_n,\frac{2n}{n+2}}_x}$ to $F^{\flat}(t)$. By compact Sobolev embedding, $F_{L_j}(t)$ converges to $F^{\flat}(t)$ strongly in localization of $C^{0}_tH^{k-1}_x$. Since $k\ge n+4$,  $F^{\flat}(t)$ solves SMCF with regularity smoother than $C^2$. Let $\varrho\in C^{\infty}_c( \Bbb R^n)$ be a test function, then (\ref{xmGGsGGs}) implies
\begin{align}
\langle \phi_{L}(t),\varrho\rangle_{L^2_x} &=\langle e^{i t \Delta} \phi_{\infty},\varrho\rangle_{L^2_x}-i\langle \int^{L}_{t}e^{i(t-s)\Delta} G(\partial^2_x\phi^{\pm}_{L}(s),\partial_x{\phi}^{\pm}_{L}(s))ds,\varrho\rangle_{L^2_x}\label{1hoo89}\\
&=\langle e^{i t \Delta} \phi_{\infty},\varrho\rangle_{L^2_x}-i\langle \int^{M}_{t} G(\partial^2_x\phi^{\pm}_{L}(s),\partial_x{\phi}^{\pm}_{L}(s)) ,e^{-i(t-s)\Delta}\varrho\rangle_{L^2_x}ds\label{2hoo89}\\
&-i\langle \int^{L}_{M}e^{i(t-s)\Delta} G(\partial^2_x\phi^{\pm}_{L}(s),\partial_x{\phi}^{\pm}_{L}(s))ds,\varrho\rangle_{L^2_x}.\label{3hoo89}
\end{align}
Note that the last line can be dominated as
\begin{align*}
&\left|\langle \int^{L}_{M}e^{i(t-s)\Delta} G(\partial^2_x\phi^{\pm}_{L}(s),\partial_x{\phi}^{\pm}_{L}(s))ds,\varrho\rangle_{L^2_x}\right|\\
&\lesssim
\|G(\partial^2_x\phi^{\pm}_{L},\partial_x{\phi}^{\pm}_{L})\|_{L^1_tL^2_x([M,L)\times\Bbb R^n)}\| \varrho\|_{L^2_x} \\
&\lesssim \| \phi_{L}\|^2_{L^{2}_tW^{k_n,\frac{2n}{n-2}}_x([M,L)\times \Bbb R^n )}C(\| \phi_{L}\|_{L^{\infty}_tH^{k}_x})\| \varrho\|_{L^2_x} \\
&\lesssim  \varepsilon_* \|e^{it\Delta}\phi_{\infty}\|^2_{_{L^{2}_tW^{k_n,\frac{2n}{n-2}}_x[M,\infty)\times \Bbb R^n}}\| \varrho\|_{L^2_x},
\end{align*}
where we applied (\ref{667000}) in the last inequality. Hence, for any $\nu>0$, there exists $M>0$ sufficiently large such that
\begin{align}\label{6cctgb}
 \sup_{L\ge M} \left|\langle \int^{L}_{M}e^{i(t-s)\Delta} G(\partial^2_x\phi^{\pm}_{L}(s),\partial_x{\phi}^{\pm}_{L}(s))ds,\varrho\rangle_{L^2_x}\right|\le \nu.
\end{align}
The second term in the RHS of (\ref{2hoo89}) splits into
\begin{align*}
 \langle \int^{M}_{t} G(\partial^2_x\phi^{\pm}_{L}(s),\partial_x{\phi}^{\pm}_{L}(s)) ,e^{-i(t-s)\Delta}\varrho\rangle_{L^2_x}ds&=\int^{M}_t\int_{|x|\le R_1}...dxds+\int^{M}_t\int_{|x|\ge R_1}...dxds\\
&:=I_1+I_2.
\end{align*}
For given $R_1>0$,   $I_2$ satisfies
\begin{align}
 |I_2|&\le \frac{1}{R_1}\int^{M}_{t} \int_{|x|\ge R_1}\|G(\partial^2_x\phi^{\pm}_{L}(s),\partial_x{\phi}^{\pm}_{L}(s))\|_{L^2_x}\|\langle x\rangle |e^{-i(t-s)\Delta}\varrho|\|_{L^2_x}ds\nonumber\\
& \le \frac{1}{R_1}\langle M-t\rangle^2\|\phi^{\pm}_{L}\|_{L^{\infty}_tH^k_x}\|\varrho\|_{H^1(\langle x\rangle dx)}\label{vvcc00}
\end{align}
where we applied Lemma \ref{Rtnnn} in the last line.
Since $F_{L_j}(t)$ converges to $F^{\flat}(t)$ strongly in localization of $L^{\infty}_tH^{k-1}_x$, for any $R_1>0,M$, $I_1$ converges to
\begin{align*}
&\langle \int^{M}_{t} \int_{|x|\le R_1} G(\partial^2_x\phi^{\flat,\pm} (s),\partial_x{\phi}^{\flat,\pm} (s)) ,e^{-i(t-s)\Delta}\varrho\rangle_{L^2_x}ds.
\end{align*}
Here, we denote $\phi^{\flat}(t)$ the associated complex valued function  with $F^{\flat}(t)$, and  $ \phi^{\flat,+}=\phi^{\flat}, \phi^{\flat,-}=\overline{\phi^{\flat}}$.

Therefore, given $t\ge 0$ and $\varrho\in C^{\infty}_c( \Bbb R^n)$, for any given sufficiently small $\nu>0$, firstly choosing $M$ sufficiently large, secondly taking $R_1>0$ large enough, for  $L\ge M$,
we infer from (\ref{1hoo89})-(\ref{3hoo89}), (\ref{vvcc00}) and  (\ref{6cctgb}) that
\begin{align*}
|\langle \phi^{\flat}(t),\varrho\rangle_{L^2_x}-\langle e^{it \Delta } \phi_{\infty},\varrho\rangle_{L^2_x}-i\langle \int^{M}_{t}e^{i(t-s)\Delta} G(\partial^2_x\phi^{\flat,\pm}(s),\partial_x{\phi}^{\flat,\pm}(s))ds,\varrho\rangle_{L^2_x}|\le \nu.
\end{align*}

Hence, we get
\begin{align*}
 \phi^{\flat}(t) =  e^{i  t\Delta} \phi_{\infty} -  i\int^{\infty}_{t}e^{i(t-s)\Delta} G(\partial^2_x\phi^{\flat,\pm}(s),\partial_x{\phi}^{\flat,\pm}(s))ds,
\end{align*}
 which together with the bounds (\ref{67000}) shows
 \begin{align*}
 \| \phi^{\flat}(t)-  e^{i  t\Delta} \phi_{\infty}\|_{L^{\infty}_tH^{k_n}_x([M,\infty)\times \Bbb R^n)}
&\lesssim   \|G(\partial^2_x\phi^{\flat,\pm}(s),\partial_x{\phi}^{\flat,\pm}(s))\|_{{L^{2}_tW^{k_n,\frac{2n}{n+2}}_x([M,\infty)\times \Bbb R^n)}}\\
&\lesssim \varepsilon^2_*  \| \phi^{ \flat}\|_{{L^{2}_tW^{k_n,\frac{2n}{n-2}}_x([M,\infty)\times \Bbb R^n)}}.
\end{align*}
 So we conclude
\begin{align*}
\lim_{t\to\infty} \| \phi^{\flat}(t)-  e^{i t \Delta} \phi_{\infty}\|_{ H^{k_n}_x}=0.
\end{align*}

 Until now, for  any given $\phi_{\infty}\in H^k$ with $\|\phi_{\infty}\|_{ H^k}\le \varepsilon_*\ll 1$, we have found an initial data $F_0$ satisfying $\|F_0-(x,0,0)\|_{H^k}\lesssim \varepsilon_*$  such that
 \begin{align*}
\lim_{t\to\infty} \| F^{n+1}(t)+\mathrm{i}F^{n+2}(t)-  e^{i t \Delta} \phi_{\infty}\|_{ H^{k_n}_x}=0.
\end{align*}
 The whole proof of Theorem 1.1 is now completed.

\section{ Supplementary materials}

\begin{lemma}\label{Rtnnn}
Let $f\in H^1(\Bbb R^n)$ with $\|f\langle x\rangle \|_{L^2_x}$. Then there holds
\begin{align*}
\|\langle x\rangle e^{it\Delta}f \|_{L^2_x}\lesssim t\|f\|_{H^{1}_x}+\|f\langle x\rangle\|_{L^2_x}.
\end{align*}
\end{lemma}
\begin{proof}
Since $\mathbb{L}_j:=(x_j+2{\mathrm{i}}t\partial_{x_j})$ commutes with $i\partial_t+\Delta$, one has by  the mass conservation law that
\begin{align*}
\|(x_j+2it\partial_{x_j})  e^{it\Delta}f \|_{L^2_x}=\|x_jf\|_{L^2}.
\end{align*}
Then the desired result follows from the fact $ \|e^{it\Delta}f \|_{H^1_x}=\|f\|_{H^1_x}$.
\end{proof}

We now recall the notion of $(r, \alpha)$-immersion originally introduced by \cite{Lan} and generalized by \cite{Bre}.  Let $\Sigma$ be an  $n$-dimensional oriented manifold. Given $q\in \Sigma$ and an immersion $F:\Sigma\to \Bbb R^{n+2}$,  let $B_q : \mathbb{R}^{n+2}\to \mathbb{R}^{n+2}$  be an Euclidean isometry which takes the origin to
$F(q)$ and maps the subspace $R^{n}\times \{0\}\subset \Bbb R^n\times \Bbb R^2$ onto the tangent bundle  $T F(\Sigma)$. Let $\pi:\Bbb R^{n+2}\to \Bbb R^n$ be the standard projection onto the first $n$ coordinates.  Define $U_{r,q}\subset \Sigma$ to be the $q$-component of the set $(\pi\circ B^{-1}_q\circ F)^{-1}(B_{r})$, where $B_{r}$ denotes the ball in $\Bbb R^n$ centered at origin of radius $r$.

  For $r,\alpha>0$, we call an immersion $F:\Sigma \to \Bbb R^{n+2}$  an $(r,\alpha)$-immersion, if for each point $q\in \Sigma$,  $  B^{-1}_q  F(U_{r,q})$ is the graph of a $C^1$ function $f:B_r
\to \Bbb R^2$  satisfying $\|\nabla f\|_{L^{\infty}}\le \alpha$.

\cite{Lan} and  \cite{Bre} proved the following result.
\begin{lemma}\label{56600}
Let $\Sigma$ be a $n$-dimensional oriented manifold,  $p > n$,  $0 < \alpha\le 1$, and assume that $F\in W^{2,p}_{loc}(\Sigma;\Bbb R^{n+2})$ is an immersion. Then $F$ is an
$(r, \alpha)$-immersion for all r satisfying
 \begin{align*}
 r^{1-\frac{n}{p}}\le \frac{ \alpha c}{\|{\bf A}\|_{H^{0,p}}},
 \end{align*}
where $c>0$ is a constant depending only on $n$.
\end{lemma}

The following was obtained by Hamilton [\cite{Hami},Section 14].
\begin{lemma}\label{2Hami}
Let $g_{ij}$ be a time dependent metric on Riemannian manifold $M$ for $0\le t<T\le \infty$. Assume that
\begin{align*}
\int^{T}_{0}\sup_{M}|\partial_t g|^2dt\le C<\infty.
\end{align*}
Then the metrics $g$ for all $t\in [0,T]$ are equivalent in the sense that
\begin{align*}
 e^{-\frac{1}{2}\int^{s_2}_{s_1}\sup_{M}|\partial_t g|^2dt}\le \frac{|X|_{g(s_2)}} {|X|_{g(s_1)}}\le e^{\frac{1}{2}\int^{s_2}_{s_1}\sup_{M}|\partial_t g|^2dt}
\end{align*}
for any $0\le s_1\le s_2\le T$ and any $X\in TM$.
\end{lemma}

The following is  Hamilton's interpolation  inequality proved in [\cite{Hami},Section 12].
\begin{lemma}\label{Hami}
Let   $T$ be any Tensor defined on manifold $\Sigma$. For   $1\le j\le i-1$, there exists a constant $C$ depending
only on dimension of $\Sigma$ and $i$, which is independent of the metric $g$ and connection such that
\begin{align}
\int_{\Sigma} |\nabla^{j} T|^{\frac{2i}{j}}d\mu\le C\max_{\Sigma}|T|^{2(\frac{i}{j}-1)}\int_{\Sigma} |\nabla^{i} T|^{2}d\mu.
\end{align}

\end{lemma}

We also need a corollary of Lemma \ref{2Hami}.
\begin{lemma}\label{3Hami}
Let $g_{ij}$ be a time dependent metric on $\Bbb R^{n}$ for $0\le t<T\le \infty$. Assume that
\begin{align*}
\int^{T}_{0}\sup_{x}|\partial_t g|^2dt\le C<\infty.
\end{align*}
Let $\rho(0)$ denote the smallest eigenvalue of $(g_{ij}(t))$ at $t=0$. Then  for all $t\in [0,T]$ there hold
\begin{align*}
 {\rm det}(g_{ij}(t)) & \ge e^{-\frac{n}{2}\int^{t}_{0}\sup_{x}|\partial_t g|^2dt} \rho^n(0)\\
 \max_{ij}|g^{ij}(t)|&\le C(g(t=0))e^{c\int^{t}_{0}\sup_{x}|\partial_t g|^2dt}
\end{align*}
for some $c>0$ depending only on $n$.
\end{lemma}
\begin{proof}
By the variational formula of eigenvalues for symmetric matrices,
Lemma \ref{2Hami} shows the smallest eigenvalue denoted by $\rho(t)$ of $(g_{ij}(t))$ satisfies
\begin{align*}
 \rho(t)\ge e^{-\frac{1}{2}\int^{t}_{0}\sup_{x}|\partial_t g|^2dt} \rho(0).
\end{align*}
Then the determinant of  $(g_{ij}(t))$ has a lower bound:
\begin{align*}
{\rm det}(g_{ij}(t)) \ge \rho^n(t)\ge e^{-\frac{n}{2}\int^{t}_{0}\sup_{x}|\partial_t g|^2dt} \rho^n(0).
\end{align*}
Then using the expression of inverse via adjoint matrix and  determinant, we conclude
\begin{align*}
|g^{ij}|\le \frac{1}{ \rho^n(t)}C_n \max_{ij}|g_{ij}|^{n-1}\le C_n\frac{\max_{ij}|g_{ij}(0)|^{n-1}}{\rho^n(0)}e^{c\int^{t}_{0}\sup_{x}|\partial_t g|^2dt},
\end{align*}
where $c>0$ depends only on $n$.
\end{proof}

\section{  The Cauchy problem of SMCF}

The local Cauchy problem of SMCF for arbitrary large data and non-compact $\Sigma$ is open.  This is the  problem what we aim to solve in this section. Let us consider the case $\Sigma=\Bbb R^n$, $\mathcal{M}=\Bbb R^{n+2}$. Since $\Bbb R^n$ is non-compact, it is convenient to set the data $F_0$ to coincide with some stationary solutions at $|x|=\infty$. This was widely adopted in the study of dispersive geometric PDEs. For SMCF, the  possibly most natural way is to use the space $\|u\|_{H^{\sigma}_{E}}$.
Let $E:\Bbb R^{n}\to \Bbb R^{n+2}$ be the plane defined in (1.2). And let
$
H^{\infty}_{E}:=\cap^{\infty}_{\sigma=0}{H^{\sigma}_{E}}.
$

Our main theorem of this part is as follows.
 \begin{theorem}\label{zz67bnm}
 Let $n\ge 2$, $\gamma_n=2[\frac{n}{2}]+2$. Given an immersion $F_0:\Bbb R^{n}\to\Bbb R^{n+2}$ belonging to $  {H}^{\infty}_{E}$,  there exists a $T>0$ depending only on  $\|{\bf A}_0\|_{H^{[\frac{n}{2}]+1,2}\cap H^{0,\gamma_n}}$  such that  SMCF has a unique smooth solution in $t\in[-T,T]$ with initial data $F_0$. Moreover, one has the bound
 \begin{align}\label{BvvnhM}
 \|F(t)\|_{C([-T,T];H^{\sigma}_E)}\le C({\sigma},T,\sup_{x,\vec{\alpha}\in \Bbb R^{n}, |\vec{\alpha}|=1} |dF_0(x)(\vec{\alpha})|^{-1},\|F_0\|_{H^{K_{\sigma}}_E}),\mbox{ }\forall \sigma\in \Bbb Z_+,
 \end{align}
 where $K_{\sigma}$ is an integer depending only on $\sigma,n$. And we have the estimates of second fundamental form:
 \begin{align}
 \|{\bf A}(t)\|_{C([-T,T];H^{\sigma,2})}&\le C(\sigma) \|{\bf A}_0\|_{H^{{\sigma,2}}},\mbox{ }\forall \sigma\in \Bbb Z_+\label{1ccM}\\
 \|{\bf A}(t)\|_{C([-T,T];H^{0,\infty})}&\le C(\|{\bf A}_0\|_{H^{{[\frac{n}{2}]+1,2}}\cap H^{0,\gamma_n}})\label{2ccM}.
 \end{align}
 Moreover, if the forward lifespan of $F(t)$ is  $0<T_*<\infty$, then there holds
 \begin{align}\label{ddmm}
 \lim_{T\to T_*}\sup_{t\in [0,T]}\|{\bf A}(t)\|_{ H^{0,\infty}  }=\infty.
 \end{align}
\end{theorem}

By density arguments, Theorem \ref{zz67bnm} also yields an existence of strong solutions if the initial data lies in sufficiently regular $H^{\sigma}_{E}$ space.
\begin{corollary}\label{z222}
Let $n\ge 2$, $\gamma_n=2[\frac{n}{2}]+2$, $s\in \Bbb Z_+$, $s\ge 2n+2$.  Given an immersion $F_0\in {H}^{s}_{E}\cap C^{s}(\Bbb R^n;\Bbb R^{n+2})$, there exists a $T>0$ depending only on $\|{\bf A}_0\|_{H^{[\frac{n}{2}]+1,2}\cap H^{0,\gamma_n}}$  such that  SMCF with initial data $F_0$ has a  solution $F\in C([-T,T];C^{\beta_s}(\Bbb R^n))$ for $ \beta_s=[\frac{1}{n}(2s-4)]-1$.  Moreover,  the continuous dependence on initial data in a geometric distance holds if $s$ is large.
\end{corollary}

\begin{remark}
Similar arguments can also give  Theorem \ref{zz67bnm}  and Corollary \ref{z222} for oriented Riemannian manifold  $\mathcal{M}$ and complete oriented Riemannian manifold  $\Sigma$ on which usual Sobolev embeddings for functions hold. We take $\Sigma=\Bbb R^{n}$ and  $\mathcal{M}=\Bbb R^{n+2}$ for simplicity.
\end{remark}

\begin{remark}
In the small data case, different from Corollary \ref{z222},   $s> \frac{1}{2}(n+5)$ is sufficient to solve the local Cauchy problem, see Lemma \ref{1xx1xxHxxjxxK} of this work.
The restriction  $s\ge 2n+2$ in  Corollary \ref{z222} is just to ensure $\beta_s\ge 3$. The other parts of proof to Corollary \ref{z222} only assume $s\ge \frac{n}{2}+3$. \end{remark}

\begin{remark}
 \cite{SGGoGGnGGgGG2}  proved the continuous dependence on initial data in a geometric distance if the corresponding second fundamental forms belong to $L^{\infty}_tH^{3,\infty}$.
\end{remark}

\begin{remark}
Although  Theorem \ref{zz67bnm} assumes smooth data, it is enough for most applications, e.g. the global regularity problem. The most important part of  Theorem \ref{zz67bnm} is the blow-up criterion (\ref{ddmm}).   In fact, one can always start with a smooth initial data and use the blow-up criterion (\ref{ddmm}) to test whether it evolves to a  global solution. This  strategy was widely adopted in the  literature of geometric flows,  and it still works even if one considers critical spaces. Note that (\ref{ddmm}) coincides   with  blow-up criterion of the mean curvature flow which is the parabolic analogy of SMCF.
\end{remark}

{\bf By discussions in Lemma 2.1, it suffices to consider $E(x)=(x,0,0)$. So in the rest of this paper, we always assume $E(x)=(x,0,0)$ for all $x\in \Bbb R^{n}$.}

\subsection{The perturbed Flow}

Given a  2-codimensional immersion   $F:\Sigma\to \Bbb R^{n+2}$,  the mean curvature vector  $\mathbf{H}$  is given by
\begin{align*}
\Delta_g F={\mathbf{H}},
\end{align*}
where $\Delta_g$ denotes the Laplacian on $\Sigma$ of the induced metric $g$   given by
\begin{align}\label{Hjk}
g_{ij} =\partial_{x_i}F\cdot\partial_{x_j}F.
\end{align}

Let us consider the perturbed flow
\begin{align}\label{ydu}
\left\{
  \begin{array}{ll}
     {\partial_t}F= J(F)\mathbf{H}(F)+\lambda \mathbf{H}(F), & \hbox{ } \\
    F(0,x)=F_0(x),\mbox{ }x\in \Sigma & \hbox{  }
  \end{array}
\right.
\end{align}
where $\lambda>0$,  $F_0\in H^{\infty}_{E}$ is an immersion.

Using the De Turck trick one can prove the local existence of solutions to (\ref{ydu}). We present it in the following lemma.  The new  problem which  might require attention is that $F_0$ lies in  $H^{\infty}_{E}$ that differs from the classic Sobolev spaces.

\begin{lemma}\label{localp}
For each $\lambda>0$, given an immersion $F_0\in H^{\infty}_{E}$, there exists $T_{\lambda}>0$ such that (\ref{ydu}) has a unique smooth  solution $u\in C([0,T_{\lambda}];H^{\infty}_{E})$ .
\end{lemma}
\begin{proof}
Consider the modified flow
\begin{align}\label{yydu}
\left\{
  \begin{array}{ll}
     {\partial_t}\tilde{F}= J(\tilde{F})\mathbf{H}(\tilde{F})+\lambda \mathbf{H}(\tilde{F})+\lambda d\tilde{F}(\tilde{g}^{ij}\tilde{\Gamma}^{k}_{ij}\frac{\partial}{\partial x_k}), & \hbox{ } \\
    \tilde{F}(0,x)=F_0(x),\mbox{ }x\in \Sigma & \hbox{  }
  \end{array}
\right.
\end{align}
where $\tilde{g}$ denotes  the metric induced by $\tilde{F}$, and  $\{\tilde{\Gamma}^{k}_{ij}\}$ denote the corresponding  connection coefficients. Define $\mathbb{R}^{n+2}$ valued function  $G$ to be
\begin{align}\label{yyydu}
G^1=\tilde{F}^1-x_1,...,G^{n}=\tilde{F}^n-x_n,\mbox{  }G^{n+1}=\tilde{F}^{n+1}, G^{n+2}=\tilde{F}^{n+2}.
\end{align}
Then $G$ solves
\begin{align}\label{yyccdu}
\left\{
  \begin{array}{ll}
     {\partial_t}G^{l}= J(\tilde{F})[\tilde{g}^{ij}\frac{\partial^2G^{l}}{\partial x_{i}\partial x_{j}}-\tilde{g}^{ij}\tilde{g}^{km}\frac{\partial^2 G^{n}} {\partial x_{i}\partial x_{j}}\frac{\partial \tilde{F}^{n}}{\partial x_{k}}\frac{\partial \tilde{F}^{l}}{\partial x_{m}}  ]+\lambda \tilde{g}^{ij}\frac{\partial^2G^{l}}{\partial x_{i}\partial x_{j}}, & \hbox{ } \\
   G(0,x)=F_0(x)-(x,0,0),\mbox{ }x\in \Sigma. & \hbox{  }
  \end{array}
\right.
\end{align}
Write the RHS of (\ref{yyccdu}) as $A^{ij}(t,x,G,D^1_x G)\partial^2_{ij}G+B(t,x,G,D^1_x G)$,
where  $A^{ij}$ is defined by
\begin{align*}
A^{ij}= J(\tilde{F})\Pi^{N}_{\tilde{F}}\tilde{g}^{ij}+\lambda  \tilde{g}^{ij}{\Bbb I}_{n+2},
\end{align*}
among which  $\Pi^{N}_{\tilde{F}}$ denotes the orthogonal projection from $\Bbb R^{n+2}$ onto the normal bundle of $\tilde{F}_t(\Sigma)$ for given $t$, and ${\Bbb I}_{n+2}$ denotes the identity matrix.
Now, let's  verify the strong parabolic hypothesis, i.e.
\begin{align}\label{xyyccdu}
 -  (A(\xi) +A^{*}(\xi)) \ge C|\xi|^2{\Bbb I}_{n+2}, \mbox{ }\forall \xi\in \Bbb R^{n},
\end{align}
where $A(\xi)$ is a $(n+2)\times (n+2)$ matrix defined by
\begin{align*}
 A(\xi) :=\sum_{i,j}A^{ij}\xi^i\xi^j.
\end{align*}
In fact, for each given $G$, the matrix $J(\tilde{F})\Pi^{N}_{\tilde{F}}$ satisfies
\begin{align*}
J(\tilde{F})\Pi^{N}_{\tilde{F}}=-(J(\tilde{F})\Pi^{N}_{\tilde{F}})^*,
\end{align*}
and there holds
\begin{align*}
\lambda \sum_{i,j}\tilde{g}^{ij}\xi^i\xi^j{\Bbb I}_{n+2}\ge C|\xi|^2,
\end{align*}
for some constant $C>0$ provided that
\begin{align}\label{U89nnn}
\|\nabla_x(G(t)-G_0)\|_{L^{\infty}_{t,x}}\le \alpha
\end{align}
for some sufficiently small constant  $\alpha>0$ depending on $G_0$.
Therefore, (\ref{xyyccdu}) holds for  some  $C>0$ if (\ref{U89nnn}) holds.

Then using standard arguments of Garding inequality and Friedrichs smoothing techniques(see e.g. Chapter 8, \cite{Tay}) and a bootstrap assumption like (\ref{U89nnn}), we deduce  that for $G_0\in H^{s}$ with $s>\frac{n}{2}+2$, there exists a unique solution $G\in C([0,T_{\lambda}],H^s(\Bbb R^n))\cap C^{\infty}((0,T_{\lambda}]\times \Bbb R^n) $ for some sufficiently small $T_{\lambda}>0$. Moreover, (\ref{U89nnn}) holds in $t\in[0,T_{\lambda}]$. And the regularity of solution propagates as long as $\|G\|_{L^{\infty}_tC^{2+\gamma}_x}<\infty$ for some $\gamma>0$. From these facts, we conclude that there exists a unique solution $G\in C([0,T_{\lambda}],H^{\infty}(\Bbb R^n))$ for some small  $T_{\lambda}>0$.

It remains to recover the solution $F$ for  (\ref{ydu}) from $G$.
Consider the ODE problem
\begin{align}\label{oyyccdu}
\left\{
  \begin{array}{ll}
\partial_t \Psi=X(\Psi) & \hbox{  }\\
\Psi(0,x)=x,\mbox{ }x\in \Sigma, & \hbox{  }
  \end{array}
\right.
\end{align}
where we define the vector filed $X:=\tilde{g}^{ij}\tilde{\Gamma}^{k}_{ij}\frac{\partial}{\partial x_k}$.
Then $F:=\tilde{F}\circ \Psi^{-1}$ solves  (\ref{ydu}). It remains to check what spaces $F$ belongs to.

For this purpose, we first bound $\partial^l_x \tilde{g}_{ij}$ and  $\partial^l_x \tilde{g}^{ij}$. Note that by  $\tilde{g}_{ij}=\partial_{i}\tilde{F}
\cdot \partial_{j}\tilde{F}$, it is easy to verify
\begin{align}\label{Vmm11}
\|\partial^l_x \tilde{g}_{ij}\|_{L^{\infty}_tL^{2}_x\cap L^{\infty}_x ([0,T_{\lambda}]\times\Bbb R^n)}\le C(l), \mbox{ }\forall \mbox{ }l\ge 1.
\end{align}
Let $\tilde{\rho}(t)$ denote the smallest eigenvalue of $(\tilde{g}_{ij}(t))$. Then (\ref{U89nnn}) implies there exists a constant $\tilde{\rho}_0>0$ independent of $t\in [0,T_{\lambda}]$ such that $\tilde{\rho}(t)\ge \tilde{\rho}_0$. Hence using the expression of inverse matrix by adjoint matrix and determinant, one has
\begin{align}\label{Vmm12}
\|\partial^l_x \tilde{g}^{ij}\|_{L^{\infty}_tL^{\infty}_x([0,T_{\lambda}]\times\Bbb R^n)}\le C(l,\tilde{\rho}_0), \mbox{ }\forall \mbox{ }l\ge 0.
\end{align}
Second, we estimate the difference of $\Psi$ and the identity map. Using  (\ref{oyyccdu}) and the expressions of $\{\tilde{\Gamma}^{k}_{ij}\}$  by $G$, one obtains by (\ref{Vmm11}) and (\ref{Vmm12}) that
\begin{align}\label{1oyyccdu}
\|\Psi-I\|_{L^{\infty}_tH^{l}_x([0,T_{\lambda}]\times\Bbb R^n)}\le T_{\lambda}C(l), \mbox{ }\forall l\ge 0.
\end{align}
Thus, for $T_{\lambda}$ sufficiently small, we get
\begin{align*}
\|\tilde{F}\circ \Psi^{-1}-(x_1,...,x_n,0,0)\|_{L^{\infty}_tH^{m}_x([0,T_{\lambda}]\times\Bbb R^n)}\le  C(m)
 \end{align*}
for any integer $m>\frac{n}{2}$.
Therefore, we conclude that $F:=\tilde{F}\circ \Psi^{-1}$ solves  (\ref{ydu}) and belongs to $C([0,T_{\lambda}];H^{\infty}_{E})$.
\end{proof}

\subsection{Review of evolutions of related geometric quantities }

Let $\Sigma$ be an $n$-dimensional oriented manifold, $(\mathcal{M},h)$ be a Riemannian manifold.
Suppose that $F :\Bbb I\times  \Sigma\to  \mathcal{M}$ is a family of  time-dependent 2-codimensional  immersions.   Recall that for each $t\in \Bbb I$, $g=g(t)$
denotes the induced metric on $\Sigma$, i.e. $g(X,Y) = h(F_*(X), F_* (Y))$ for $X,Y\in T\Sigma$.  Denote the associated  Riemannian  connection and volume form on $(\Sigma,g)$
by $\nabla $ and $d\mu=d\mu(t)$ respectively.
The second fundamental form ${\bf A}$ is defined by ${\bf A}(X,Y)=  D_{X}F_*(Y)  -F_*(\nabla_{X}Y)$, where $D$ denotes the induced connection on $F^*T\mathcal{M}$. The mean curvature
vector ${\bf H}$ is given by the trace ${\bf H} =\sum_{i}{\bf A}(e_i,e_i)$, where $\{e_i\}$ is any orthonormal  basis of $T\Sigma$. And we have the normal connection $\nabla^{\mathfrak{N}}$ defined by  $\nabla^{\mathfrak{N}}_{X}\eta = (\nabla_{X}\eta)^{\bot}$, i.e. the projection of $\nabla_{X}\eta$ onto the normal bundle $N F(\Sigma)$, where $\eta$ is a normal vector field along $F$.  For simplicity, we denote all these connections by $\nabla$.

The energy estimate  of \cite{SGGoGGnGGgGG1,SGGoGGnGGgGG2} has been recalled in (\ref{xv37}) of Lemma 3.1. Along the perturbed flow,  (\ref{xv37}) also holds.
We also need the evolution equations of induced metrics and volumes.
\begin{lemma}[\cite{SGGoGGnGGgGG1,SGGoGGnGGgGG2}]\label{HjnMLmm}
Let $g_{ij}=\frac{\partial F}{\partial x_i}\cdot \frac{\partial F}{\partial x_i}$ be the induced metric under the coordinates $(x_1,...,x_n)$ for $\Sigma=\Bbb R^n$. And denote ${\bf A}_{ij}={\bf A}(\frac{ \partial}{\partial x_i},\frac{ \partial}{\partial x_j})$.  Then
along the perturbed flow, one has
\begin{align}\label{7.7x}
\frac{d}{dt}g_{ij}=-2\langle J{\bf H},A_{ij} \rangle-2\lambda\langle {\bf H},A_{i,j} \rangle
\end{align}
and
\begin{align*}
\frac{d}{dt}d\mu=-\lambda |{\bf H}|^2d\mu.
\end{align*}
\end{lemma}

\subsection{ Blow-up criterion and uniform Sobolev norm estimates} \label{67qq}

The main result of this subsection is the following blow-up criterion for  the perturbed flow (\ref{ydu}).
\begin{proposition}\label{77788}
If $F_0\in H^{\infty}_{E}$ and the corresponding solution $F(t)$ to perturbed flow (\ref{ydu}) in $t\in [0,\tau]$ satisfies
\begin{align}\label{uguiin}
\sup_{t\in[0,\tau]}\|{\bf A}(t)\|_{H^{0,\infty}}\le D,
\end{align}
then for each $j\ge 0$  there holds
\begin{align}\label{Hjkk}
\|F \|_{C([0,\tau];H^{j}_{E})}\le C_1(j,F_0)
\end{align}
for some constant $C_1(j,F_0)>0$ depending only on   $j,n,D,\rho_0,\|\partial_xF_0\|_{L^{\infty}_x}$, $\sum^{L}_{j=0}\|\partial^j_x\Gamma_0\|_{L^P_x\cap L^2_x}$ and $\|{\bf A}_0\|_{H^{L,2}}$ with $L,P$ depending only on $n,j$. Here, $\rho_0$ denotes the smallest eigenvalue of $(g_{ij}(t))$ at $t=0$, $\Gamma_0$ denotes the  Christoffel symbols   associated with $F_0$.
\end{proposition}

The proof of  Proposition \ref{77788} will be divided into several lemmas.

First, we bound second fundamental forms.
\begin{lemma}
Assume that  $F_0\in H^{\infty}_{E}$ and the corresponding solution $F(t)$ to the perturbed flow (\ref{ydu}) in $t\in [0,\tau]$ satisfies (\ref{uguiin}). Then for any $l\in \Bbb N$
\begin{align}\label{2uguiin}
\sup_{t\in [0,\tau] } \|{\bf A}(t)\| _{H^{l,2}_x}\lesssim   \|{\bf A}(0)\|_{H^{l,2}_x}.
\end{align}
Moreover, for each $l\ge 0,p\ge 2$, there holds
\begin{align}\label{3uguiin}
\sup_{t\in [0,\tau] } \|{\bf A}(t)\| _{H^{l,p}_x}\lesssim  C(\|{\bf A}(0)\|_{H^{L_{l,p},2}_x}),
\end{align}
where we define
\begin{align}\label{LL}
L_{l,p}=[ \frac{lp}{2}]+1.
\end{align}
The implicit constants in (\ref{2uguiin}) and (\ref{3uguiin})  depend  only on $D,n,l,p$.
\end{lemma}
\begin{proof}
By (\ref{uguiin}) and  (\ref{xv37}), we deduce from Gronwall inequality that
\begin{align*}
\sup_{t\in [0,\tau] } \|{\bf A}(t)\| _{H^{l,2}_x}\lesssim \exp(\tau \||{\bf A}|\|^2_{L^{\infty}_{t,x}}) \|\mathbf{A}(0)\|_{H^{l,2}}
\end{align*}
for any integer  $l\ge 0$. So (\ref{2uguiin}) follows by (\ref{uguiin}).
The remained (\ref{3uguiin}) follows by (\ref{2uguiin}) and Hamilton's interpolation inequality (see (\ref{Hami})).
\end{proof}

Second, we  bound  the induced  metrics. (\ref{7.7x}) shows
\begin{align*}
 | \partial_t g| \le \|{\bf A}\|^2_{H^{0,\infty}}.
\end{align*}
Then, by (\ref{uguiin}), Lemma \ref{2Hami}, we see $(g_{ij}(t))$ are equivalent to each other for different  $t\in [0,\tau]$
in the sense that
\begin{align}\label{R5t6}
\left| \ln \frac{| X_{g(t)}|}{|X_{g(0)}|} \right|\le C(n) \tau D^4, \mbox{ }\forall X\in T\Sigma,
\end{align}
for some $C(n)>0$ depending only on $n$.

Particularly,  the volume $d\mu(t)$ is comparable to the standard Euclidean measure $dx$ with constants depending only on $n,D,\rho_0,\|\partial_xF_0\|_{L^{\infty}_{x}}$, where $\rho_0$ denotes the smallest eigenvalue of $(g_{ij}(t))$ at $t=0$. So in the following,  the norms    $L^{p}(\Bbb R^n,dx)$ and $L^{p}(\Bbb R^n,d\mu)$ are all denoted by $L^p_x$ for convenience. It suffices to keep in mind  that an implicit constant depending only    $n,D,\rho_0,\|\partial_xF_0\|_{L^{\infty}_{x}}$  occurs when one changes from one norm to the other.

Since $g_{ij}=\partial_iF\cdot \partial_jF$, (\ref{R5t6}) shows
\begin{align}\label{22xxfg}
 \|\partial_x F\|_{L^{\infty}_{t,x}([0,\tau]\times \Bbb R^n)} \le \|\partial_xF_0\|_{L^{\infty}_x}e^{C(n) \tau D^4}.
\end{align}
And directly applying the perturbed flow equation yields
\begin{align*}
\frac{1}{2}\frac{d}{dt} \|F(t)-E\|^2_{L^2_x}\le \|\mathbf{H}\|_{L^{\infty}_{t,x}}\|F(t)-E\|_{L^2_x}\le D\|F(t)-E\|_{L^2_x},
\end{align*}
which gives the
bound  of $\|F-E\|_{L^2_x}$:
\begin{align}\label{fff}
 \|F(t)-E\|^2_{L^{\infty}_t[0,\tau]L^2_x}\le  C(D,\|F_0\|_{L^2_E}).
\end{align}

Moreover, we introduce a notation for convenience. Given  a tensor ${\Bbb T}$ on $\Sigma_t$, let ${\Bbb T}^{\alpha_1,..,\alpha_s}_{\beta_1,..,\beta_r}$ denote its components under the standard coordinates $(x_1,...,x_n)$ of $\Bbb R^n$, denote
\begin{align}\label{hhhoo}
\partial^l_{x_j} {\Bbb T}=   (\partial^l_{x_j} \mathbb{T}^{\alpha_1,..,\alpha_s}_{\beta_1,...,\beta_r})\frac{\partial}{\partial x^{\alpha_1}}\otimes...\otimes\frac{\partial}{\partial x^{\alpha_r}}\otimes dx^{\beta_1}\otimes ...\otimes dx^{\beta_r},
\end{align}
We remark that  the Christoffel symbols $\Gamma$ can be viewed as a tensor by subtracting a reference connection. Since $\Sigma_t$ has  natural global  coordinates $(x_1,...,x_n)$, one can put the reference connection to be the trivial one. Thus in the following we view  Christoffel symbols $\Gamma$ as a tensor.

Third, we bound $\|\partial^l_x F\|_{L^2_x}$ with $l\ge 2$.
The rest of this subsection closely follows from Section 4 of Kuwert-Sch\"atzle \cite{KS} with some refinements. We shall need to interchange the derivatives of multilinear forms on $\Sigma$ having
normal values along $F$. If $\eta ,\eta'$ are forms of this type, we denote by $\eta*\eta'$
any normal-valued, multilinear form depending on $\eta$ and $\eta'$ in a universal,
bilinear way.

The equivalence of   $(g_{ij}(t))$ and the identity $\nabla J=0$  will be frequently used in the rest of this subsection  without emphasis.
By Lemma \ref{HjnMLmm}, (\ref{uguiin}) and (\ref{3uguiin}), for $p\ge 2,l\ge 0$ one has
\begin{align*}
\|\nabla^{l} \partial_tg\|_{L^{\infty}_tL^{p}_x\cap L^2_x([0,T^*_{\lambda}]\times\Bbb R^{n})}\le C(l,p,D,\|{\bf A}_0\|_{H^{L_{l,p},2}_x}),
\end{align*}
where $L_{l,p}$ is defined in (\ref{LL}).
And the derivatives of the Levi-Civita connection, i.e. the Christoffel symbols, are schematically written as
\begin{align}\label{5r}
\partial_t \nabla =\nabla {\bf A}*S_{\lambda}{\bf H}+{\bf A}*{S_{\lambda}}\nabla{\bf H},
\end{align}
where we denote $S_{\lambda}=J+\lambda I$ for simplicity.
So we again obtain  for $p\ge 2,l\ge 0$  that
\begin{align}\label{Ghvbmmmm000}
\|\nabla^{l}(\partial_t \nabla)\|_{L^{\infty}_tL^p_x\cap L^2_x} \le C(l,p,D,\|{\bf A}_0\|_{H^{L_{l+1,2p},2}_x}).
\end{align}
For any tensor $\mathbb{T}$, one has
\begin{align*}
\nabla^{l} \mathbb{T}=\partial^{l} \mathbb{T}+\sum^{l}_{m=1}\sum_{l_1+...+l_{m}+k\le l-1} \partial^{l_1}\Gamma...\partial^{l_{m}}\Gamma\cdot \partial^{k} \mathbb{T}.
\end{align*}
Hence, by induction, for $p\ge 2$ we get
\begin{align}\label{jnkmL}
\|\partial^{l} \mathbb{T}\|_{L^{p}_{x}}\le \|\nabla^l \mathbb{T}\|_{L^p_x}+C(\|\Gamma\|_{L^{2}_x\cap L^{l^2p}_x},...\|\partial^{l-1}\Gamma\|_{L^2_x\cap L^{\frac{l^2p}{l-1}}_x})  (\|\partial^{l-1}\mathbb{T}\|_{L^2_x\cap L^{\frac{lp}{l-1}}_x}+...\|\mathbb{T}\|_{L^2_x\cap L^{lp}_{x}}),
\end{align}
where  the implicit constant depends only on $l,n,p,D,\rho_0,\|\partial_xF_0\|_{L^{\infty}_x}$.

The following lemma bounds $\partial^l\Gamma$.
\begin{lemma}
Assume that (\ref{uguiin}) holds. Then there hold
\begin{align}
\|\partial^{l}(\partial_t\Gamma) \|_{L^{\infty}_tL^{p}_{x}([0,\tau]\times \Bbb R^n)}&\le   C(l,p,F_0),\mbox{ }\forall l\ge 0\label{4jnkmL}\\
\|\partial^{l}\Gamma\|_{L^{\infty}_tL^{p}_{x}([0,\tau]\times \Bbb R^n)}&\le   C(l,p,F_0),\mbox{ }\forall l\ge 0,\label{3jnkmL}
\end{align}
where  the implicit constant $C(l,p,F_0)$ depends only on
$$l,n,p,D,\rho_0,\|\partial_xF_0\|_{L^{\infty}_x}, \sum^{L}_{j=0}\|\partial^j\Gamma_0\|_{L^2_x\cap L^{P}_x}, \|{\bf A}_0\|_{H^{L,2}}$$
with $L,P$ depending only on $n,p,l$.
\end{lemma}
\begin{proof}
By (\ref{Ghvbmmmm000}),  one obtains bounds of $\nabla^{l}(\partial_t\Gamma)$ for $l\ge 0$. Integrating  (\ref{5r}) with respect to $t$, by (\ref{uguiin})-(\ref{3uguiin}), we get
\begin{align}
\| \Gamma(t) \|_{L^{\infty}_tL^{p}_{x}([0,\tau]\times \Bbb R^n)}&\lesssim_{p,\tau} \|\Gamma(0)\|_{L^p_x}+ C(\|{\bf A}_0\|_{H^{L_{1,2p},2}}).
\end{align}
Then  (\ref{4jnkmL}), (\ref{3jnkmL}) follow by
applying (\ref{jnkmL}) to $T=\partial_t \Gamma$, the inequality
\begin{align*}
\frac{d}{dt}\|\partial^{l}\Gamma\|^p_{L^p_x}\lesssim  \|\partial^{l}\partial_t\Gamma\|^p_{L^p_x}+ \|\partial^{l}\Gamma\|^p_{L^p_x},
\end{align*}
(\ref{3uguiin}),  (\ref{Ghvbmmmm000}) and induction.
\end{proof}

Now, we are ready to deal with bounds of $|\partial^{l}F|$ and $|\partial^{l}\nabla A|$ in $L^p$ spaces.  \cite{KS} proved the point-wise inequality
\begin{align}
&\partial^m\nabla^{i}{\bf A}-\nabla^{j+1}{\bf A}=\sum^{m}_{l=1}\partial^{l-1}(\nabla^{m+i-l}{\bf A}*{\bf A}*\partial F), \mbox{ }m+i=j+1\\
&\partial^2F={\bf A}+\partial F\cdot \Gamma.\label{56tgvbbn}
\end{align}
Then we obtain by (\ref{3jnkmL}), (\ref{3uguiin}),   (\ref{22xxfg}) and induction that
\begin{align}
\|\partial^{l}F\|_{L^{\infty}_tL^p_x([0,\tau]\times \Bbb R^n)}&\le  C(l,p,F_0), \mbox{ }\forall l\ge 2,\mbox{ } p\ge 2\label{nnFvbzzz}\\
\|\partial^{l}\nabla^{m}{\bf A} \|_{L^{\infty}_tL^p_x([0,\tau]\times \Bbb R^n)}&\le  C(l,p,m,F_0), \mbox{ }\forall l,m\ge 0,\mbox{ } p\ge 2.
\end{align}
 (We remark that the $\partial F$ term in the RHS of (\ref{56tgvbbn}) only belongs to $L^{p}_{t,x}$ with $p=\infty$.)

Therefore, (\ref{nnFvbzzz}) now leads to
\begin{align*}
\|F\|_{C([0,\tau];{\dot H}^{j}_{E})}\le    C_1(j,F_0),\mbox{ }\forall \mbox{ }j\ge 2,
\end{align*}
for some constant $C_1(j,F_0)>0$ depending only on
$$j,n,D,\rho_0,\mbox{ } \|\partial_xF_0\|_{L^{\infty}_x},\mbox{ } \sum^{L}_{l=0}\|\partial^l\Gamma_0\|_{L^p_x\cap L^2_x},\mbox{ }   \|{\bf A}_0\|_{H^{L,2}},$$
with some $L,p$ depending only on $n,j$.
Using (\ref{fff}) to bound $\|F\|_{L^2_{E}}$, we finally get
 \begin{align}\label{nnnnFvbzzz}
\|F\|_{C([0,T_{\lambda}];{H}^{j}_{E})}\le C_1(j,F_0),\mbox{ }\forall \mbox{ }j\ge 0.
\end{align}
This proves Proposition \ref{77788}.

\subsection{Proof of Theorem 7.1}

Using Lemma \ref{56600}, one can prove the following result:
\begin{lemma}\label{se2}
If $F_0\in H^{\infty}_{E}$ and the corresponding solution $F(t)$ to perturbed flow (\ref{ydu}) in $t\in [0,\tau]$ satisfies
 \begin{align*}
 \sup_{t\in [0,\tau]}{\|{\bf A}\|_{H^{0,p}}}\le \theta,
 \end{align*}
for some $p>n$, $\theta>0$. Fix $\alpha\in (0,1)$. Then there exists a constant $r>0$ depending only on $\alpha,\theta,n$ so that the following holds: Given $t\in [0,\tau]$, there exist
 a covering of $\Bbb R^n$ by   $\{U_{r,q_i}\}^{\infty}_{i=1}$,  and a sequence of isometries $\{B_{i}(t)\}^{\infty}_{i=1}$ in $\Bbb R^{n+2}$, so that  $B^{-1}_i\circ F(t)(U_{r,q_i})$ is a graph of a differentiable  function $f_i:B_{r}\subset \Bbb R^{n}\to \Bbb R^{2}$ with
\begin{align}
f_i(t,0)=0, \mbox{ }Df^{i}(t,0)=0,  \|D f_i\|_{C^0(B_{r})}\le \alpha   \mbox{ }\forall i\ge 0.\label{CfgbJ}
\end{align}
\end{lemma}

We also need a localized version of Lemma 3.4 in our previous work \cite{Li}.
\begin{lemma}\label{se3}
Given $r>0$, assume that $u:B(0,r)\subset \Bbb R^{n} \to \Bbb R^{n+2}$ is a graph, $u(x)=(x,u^1(x),u^2(x))$, $x\in B(0,r)$. Denote $B_1(y)$ to be the ball of radius $1$  in $\Bbb R^{n+2}$ centered at $y\in\Bbb R^{n+2}$. Let $\Sigma=u(B(0,r))$, and denote the induced metric, measure, the second fundamental form, the mean curvature vector  associated with  the immersion  $u$ by $g$, $\mu$, ${\bf A}$, and ${\bf H}$ respectively. View  $\mu$ as a measure in $\Bbb R^{n+2}$ with support in $\Sigma$. If
$\mu(B_1(\tilde{y})\bigcap \Sigma)+\|{\bf H}\|_{L^{n+\gamma}(d\mu)}\le D$ for any $\tilde{y}\in \Bbb R^{n+2}$ and some $D>0,\gamma>0$, then given $p>n$, for every covariant  tensor ${\bf T}$ there holds
\begin{align*}
\max_{\Sigma}|{\bf T}|_{g}\le C\left((\int_{\Sigma}|\nabla {\bf T}|^p_gd\mu)^{\frac{1}{p}}+(\int_{\Sigma}| {\bf T}|^p_gd\mu)^{\frac{1}{p}}\right),
\end{align*}
where $C$ depends only on $d,p,\gamma,D$.
\end{lemma}

Combining Lemma \ref{se2} with Lemma \ref{se3}, we get
\begin{lemma}\label{se4}
If $F_0\in H^{\infty}_{E}$ and the corresponding solution $F(t)$ to the perturbed flow in $t\in [0,\tau]$ satisfies
 \begin{align}\label{ddgh89}
 \sup_{t\in [0,\tau]}[{\|{\bf A}(t)\|_{H^{0,p}}} +{\|{\bf H}(t)\|_{H^{0,p}}}]\le \theta,
 \end{align}
for some $p>n$, $\theta>0$. Then for $\Sigma_t=F(t)(\Sigma)$ and   every covariant  tensor ${\bf T}$ there holds
\begin{align}\label{ddgfh89}
\max_{\Sigma_t}|{\bf T}|_{g}\le \mathcal{C}(n,p,\theta) \left((\int_{\Sigma_t}|\nabla {\bf T}|^p_gd\mu)^{\frac{1}{p}}+(\int_{\Sigma_t}| {\bf T}|^p_gd\mu)^{\frac{1}{p}}\right),
\end{align}
where $\mathcal{C}$ depends only on $n,p,\theta$.
\end{lemma}
 \begin{proof}
 Since Lemma \ref{se4} aims to bound the $L^{\infty}$ norm (rather than $L^p$ norms with finite $p$), it is convenient   to work in a small neighborhood  of any given point on $\Sigma_t$. Let $\alpha\in (0,1)$ be given. Lemma \ref{se2} has shown that for given $t\in [0,\tau]$, $x\in \Sigma$, there exists an open set $U_{r,\alpha}\subset \Sigma $ containing $x$  and an Euclidean isometry $B$ such that $ B^{-1}\circ F(U_{r,\alpha})\subset \Bbb R^{n+2}$ is a graph of some differential function $f:B(0,r)\subset \Bbb R^n\to \Bbb R^{2}$ with $\|D f\|_{C^0(B(0,r))}\le \alpha$. Since ${\Bbb T}$ is a covariant  tensor, and (\ref{ddgh89}), (\ref{ddgfh89}) are   independent of coordinates, it suffices to assume $F(U_{r,\alpha})$ is of the form
$$\{B  (y,f^1(y),f^2(y)):  |y|<r\}.$$
 By (\ref{CfgbJ}), the induced metric $g_{ij}$ in the new coordinates $y$ satisfies
 \begin{align*}
 |g_{ij}|\le 2.
 \end{align*}
 Thus $\mu(F(U_{r,\alpha})\cap B_1(\tilde{y}))\le C(r,n)$ for all $\tilde{y}\in \Bbb R^{n+2}$. Then applying Lemma \ref{se3} to $F(U_{r,\alpha})$, one obtains the bound of $|{\Bbb T}|_{g}$ at the given points $t$ and $x$. Since the LHS of (\ref{ddgfh89}) is $L^{\infty}$ norm,
  Lemma \ref{se4} follows by repeating  the above arguments for every $t$ and $x$.
\end{proof}

Let $T_{\lambda}>0$ be the lifespan of the perturbed flow.  Let $T^{*}_{\lambda}>0$  be the maximal time in $[0, \min(1,,T_{\lambda})]$ such that
\begin{align}\label{xxfg}
\sup_{t\in [0,T^{*}_{\lambda}]}\|{\bf A}(t)\|_{H^{0,\infty}_x}\le C_0,
\end{align}
where $C_0>0$ is a constant to be determined later.

By Lemma \ref{se4} here and some modifications of  arguments in  Lemma 3.9 of our previous work \cite{Li}, one can prove
\begin{lemma}\label{98}
Assume that  $F_0\in H^{\infty}_{E}$ and (\ref{xxfg}) holds. For well chosen constants $C_0,T_0$  depending only on $n$ and $\|{\bf A}_0\|_{H^{[\frac{n}{2}]+1,2}},\|{\bf A}_0\|_{H^{0,\gamma_n}}$, $\gamma_n=2[\frac{n}{2}]+2$,   the solution of perturbed flow (\ref{ydu}) satisfies
\begin{align}\label{uixxjjn}
\sup_{t\in[0,T]}\|{\bf A}(t)\|_{H^{0,\infty}}\le  \frac{1}{2}C_0.
\end{align}
\end{lemma}
\begin{proof}We sketch the proof. Use (\ref{xv37}) and (\ref{xxfg}) to get
\begin{align}\label{KK000}
\sup_{t\in[0,T^*_{\lambda}]}\|{\bf A}(t)\|_{H^{[\frac{n}{2}]+1,2}}\le   \|{\bf A}_0\|_{H^{[\frac{n}{2}]+1,2}}e^{T^*_{\lambda}C^2_0}.
\end{align}
By the equation
\begin{align}
\frac{d}{dt}\|{\bf A}\|^{\gamma}_{H^{0,\gamma}  }\le C(\gamma) (\|\nabla {\bf A}\|^{\gamma}_{H^{0,\gamma}}+\|{\bf A}\|^2_{H^{0,\infty}}\| {\bf A}\|^{\gamma}_{H^{0,\gamma}}),
\end{align}
 and the bound (by Hamilton's interpolation inequality)
\begin{align}\label{gbb}
 \|\nabla {\bf A}\|^{2[\frac{n}{2}]+2}_{H^{0,2[\frac{n}{2}]+2}_x}\lesssim \|  {\bf A}\|^{2}_{H^{[\frac{n}{2}]+1,2}} \| {\bf A}\|^{2[\frac{n}{2}]}_{H^{0,\infty}},
\end{align}
we obtain for $\gamma_n=2[\frac{n}{2}]+2>n$ that
\begin{align}
 \|{\bf A}\|^{\gamma_n}_{H^{0,\gamma_n} }\le
  C_ne^{T^*_{\lambda}C^2_0}\|{\bf A}_0\|^{\gamma_n}_{H^{0,\gamma_n} }+T^{*}_{\lambda}e^{T^*_{\lambda}C^2_0}\left( C_n C^{2[\frac{n}{2}]}_0 \|{\bf A}_0\|^2_{H^{[\frac{n}{2}]+1,2}}e^{2T^*_{\lambda}C^2_0}\right)
\end{align}
where we applied (\ref{KK000}), (\ref{xxfg}) to bound the RHS of (\ref{gbb}) while using Gronwall inequality.
Assume that   $T^{*}_{\lambda}$ satisfy
\begin{align}\label{hbnjjP}
T^{*}_{\lambda}C^2_0+T^{*}_{\lambda}\left( C_n C^{2[\frac{n}{2}]}_0 \|{\bf A}_0\|^2_{H^{[\frac{n}{2}]+1,2}}e^{2T^*_{\lambda}C^2_0}\right)&\le \frac{1}{4}.
\end{align}
Then  Lemma \ref{se4} shows
\begin{align*}
 \sup_{t\in [0,T^{*}_{\lambda}]}\|{\bf A}(t)\|_{H^{0,\infty}}&\le \mathcal{C}(n,\gamma_n,2C_n\|{\bf A}_0\|_{H^{0,\gamma_n}}+1)( \|\nabla {\bf A}\|_{H^{0,\gamma_n}}+ \|  {\bf A}\|_{H^{0,\gamma_n}})\\
  &\le \mathcal{C}(n,\gamma_n,2C_n\|{\bf A}_0\|_{H^{0,\gamma_n}}+1)( C_n\|  {\bf A}\|^{\frac{2}{\gamma_n}}_{H^{[\frac{n}{2}]+1,2}} \| {\bf A}\|^{\frac{2[\frac{n}{2}]}{\gamma_n}}_{H^{0,\infty}}),
\end{align*}
which further gives
\begin{align*}
 \sup_{t\in [0,T^{*}_{\lambda}]}\|{\bf A}(t)\|_{H^{0,\infty}}
& \le \left[C_n\mathcal{C}(n,\gamma_n,2C_n\|{\bf A}_0\|_{H^{0,\gamma_n}}+1)\right]^{\frac{\gamma_n}{2}}  \|  {\bf A}\|_{H^{[\frac{n}{2}]+1,2}}  \\
& \le \left[C_n\mathcal{C}(n,\gamma_n,2C_n\|{\bf A}_0\|_{H^{0,\gamma_n}}+1)\right]^{\frac{\gamma_n}{2}}\|{\bf A}_0\|_{H^{[\frac{n}{2}]+1,2}}e^{T^*_{\lambda}C^2_0}.
\end{align*}
Therefore, taking $C_0>0$ to be sufficiently large  to verify
\begin{align*}
C_0\ge 2 \left[C_n\mathcal{C}(n,\gamma_n,2C_n\|{\bf A}_0\|_{H^{0,\gamma_n}}+1)\right]^{\frac{\gamma_n}{2}}e^2\|{\bf A}_0\|_{H^{[\frac{n}{2}]+1,2}},
\end{align*}
and letting $T  >0$ be sufficiently small to fulfill
\begin{align*}
 T C^2_0+
T \left( C_n C^{2[\frac{n}{2}]}_0 \|{\bf A}_0\|^2_{H^{[\frac{n}{2}]+1,2}}e^{2 C^2_0}\right)&\le \frac{1}{4},
\end{align*}
(note that this corresponds to  (\ref{hbnjjP})) we see
\begin{align*}
&\sup_{t\in [0,T]}\|{\bf A}(t)\|_{H^{0,\infty}}\le \frac{1}{2}C_0.
\end{align*}
So  (\ref{uixxjjn})  follows, and $T$ depends only on $\|{\bf A}_0\|_{H^{0,\gamma_n}}$, $\|{\bf A}_0\|_{H^{[\frac{n}{2}]+1,2}}$ as desired.
\end{proof}

Let us review what has been done.
Let $\lambda\in (0,1)$. Consider the flow (\ref{ydu}).
Define $T^*_{\lambda}\in [0,\min(1,T_{\lambda})]$ to be the maximal time such that (\ref{xxfg}) holds.
We have shown $T^*_{\lambda}\ge T$ for some $T,C_0>0$  depending only on $n$, $\|{\bf A}_0\|_{H^{[\frac{n}{2}]+1,2}}$, $\|{\bf A}_0\|_{L^{\gamma_n}_x}$. Then  by Proposition \ref{77788}  and  (\ref{uixxjjn}) one has
\begin{align}\label{nnFvb}
\|F_{\lambda}\|_{C([0,T];H^{j}_{E})}\le C(F_0,T,j)\mbox{ }\forall j\ge 0.
\end{align}
Then selecting a subsequence $\lambda_m\to 0$, there exists a smooth map $F\in H^{\infty}_{E}$ such that $F_{\lambda_m}\to F$ locally in $C([0,T];H^{j}_{E})$ for any $j\in\Bbb Z_+$.
So $F$ solves SMCF with initial data $F_0$. The continuous dependence on initial data in geometric distance  has been proved in \cite{SGGoGGnGGgGG2}.

\subsection{Improved uniform bounds and Non-smooth Data}

In this subsection, we first prove  the bound  (\ref{BvvnhM}).
\begin{lemma}\label{88}
Let $L\in \Bbb Z_+$, $P\ge 2$, and $F_0\in H^{L+[\frac{n}{2}]+1}_{E}$ be a 2-codimensional  immersion. Then
\begin{align*}
&\sum^{L}_{j=0}\|\partial^j\Gamma_0\|_{L^P_x\cap L^2_x}+ \|F_0\|_{L^2_{E}}+\|\partial_x F_0\|_{L^{\infty}_x}+\|{\bf A}_0\|_{ H^{L,2}}\\
& \le C(\|F_0\|_{H^{L+[\frac{n}{2}]+1}_{E}},\sup_{x\in\Bbb R^n} \max_{\vec{\alpha}\in \Bbb R^{n}, |\vec{\alpha}|=1} |dF_0(\vec{\alpha})|^{-1}).
\end{align*}
\end{lemma}
\begin{proof}
Let $g$ denote the induced metric.
Observe that the  $\{ \partial^{l}_xg_{ij}\}^{L}_{l=0} $ and $\{\partial^l_x F_0\}^{L}_{l=0}$ involved parts are easy to dominate by $ \|F_0\|_{H^{L+[\frac{n}{2}]+1}_{E}}$. The relatively troublesome part is to bound $\| \partial^{l}g^{ij}\|_{L^{\infty}_x}$ for $l=0,1,...,L$.
Using the expression of inverse matrix via adjoint matrix and  determinant, it reduces to bound
\begin{align*}
\frac{1}{{\rm det}((g_{ij}))}.
\end{align*}
Let $\rho$ denote the smallest eigenvalue of $(g_{ij})$, and denote $\frac{\partial F_0}{\partial \vec{\alpha}}=\sum^n_{i=1}\alpha_i\partial_iF_0$, then
\begin{align*}
\rho&=\min_{\vec{\alpha}\in \Bbb R^{n}, |\vec{\alpha}|=1}\sum_{i,j}\alpha_ig_{ij}\alpha_j=\min_{\vec{\alpha}\in \Bbb R^{n}, |\vec{\alpha}|=1}|\frac{\partial F_0}{\partial \vec{\alpha}}|^2.
\end{align*}
Since $F_0$ is an immersion, the kernel of $dF_0(x)$  is $\{0\}$. Thus there exists $\beta(x)>0$ such that
\begin{align}
\min_{\vec{\alpha}\in \Bbb R^{n}, |\vec{\alpha}|=1}|\frac{\partial F_0}{\partial \vec{\alpha}}|^2\ge \beta^2(x). \label{y7899}
\end{align}
Let $\imath:\Bbb R^{n}\to \Bbb R^{n+2}$ denote the immersion $\imath(x)=(x,0,0)$. By $F_0\in H^{[\frac{n}{2}]+2}_{E}$, there exists $R>0$ sufficiently large such that
\begin{align}\label{0tt11}
\|dF_0-\imath\|_{L^{\infty}_{x}(|x|\ge R)}\le \frac{1}{100},
\end{align}
from which one obtains
\begin{align}
\inf_{|x|\ge R}\min_{\vec{\alpha}\in \Bbb R^{n}, |\vec{\alpha}|=1}|\frac{\partial F_0}{\partial \vec{\alpha}}|^2\ge \frac{1}{4}. \label{y9899}
\end{align}
Thus applying (\ref{y7899}) in $\{x\in \Bbb R^n:|x|\le R\}$ and  (\ref{y9899}) in  $\{x\in \Bbb R^n:|x|\ge R\}$ respectively, we infer from the continuity of $dF_0$ that
\begin{align}
\inf_{x\in \Bbb R^n}\min_{\vec{\alpha}\in \Bbb R^{n}, |\vec{\alpha}|=1}|\frac{\partial F_0}{\partial \vec{\alpha}}|^2\ge \beta^2 \label{y1099}
\end{align}
for some $\beta>0$. Therefore, one has
\begin{align*}
\frac{1}{{\rm det}((g_{ij}))}\le \frac{1}{\rho^n}\le \beta^{-2n}.
\end{align*}
This  together with Sobolev embeddings dominates $\partial^l_x\Gamma_0$. And writing   $\nabla ^{l}{\bf A}_0$ via $\partial^{m}_xF_0$ and $\partial^{j}_x\Gamma_0$ yields  bounds of $\|{\bf A}_0\|_{H^{l,2}}$.    So our lemma follows.
\end{proof}

Hence,  (\ref{BvvnhM}) follows by Proposition \ref{77788}, Lemma \ref{98} and Lemma  \ref{88}. And Theorem 7.1 has been done.

\subsection{Proof of  Corollary \ref{z222}}
Finally, let us consider non-smooth data.
Given $L\ge [\frac{n}{2}]+2$,  let $F_0\in H^{L}_{E}$ be a 2-codimensional immersion. We aim to find a family of smooth immersions $F^{\delta}_{0}$ such that
\begin{align}
 \sup_{\delta\in (0,1)}\|F^{\delta}_0\|_{ H^{L}_{E}}&\le C(F_0)\label{8ujj}\\
 \sup_{\delta\in (0,\delta_*)}\sup_{x\in \Bbb R^n} \max_{\vec{\alpha}\in \Bbb R^{n}, |\vec{\alpha}|=1} |dF^{\delta}_0(\vec{\alpha})|^{-1}&\le C(F_0),\label{9ujj}
\end{align}
where $\delta_*>0$ is some sufficiently small constant.

Recall $\imath(x)=(x,0,0)$.
Let $F^{\delta}_{0}$ be $F_0*\eta_{\delta}$, where $\eta_{\delta}=\delta^{-n}\eta(\frac{x}{\delta})$, and $\eta$ is the standard modifier in $\Bbb R^n$. Then
\begin{align*}
\|F^{\delta}_0\|_{H^{l}_{E}}&\le \|F_0*\eta_{\delta}-\imath\|_{H^{l}}\le \|(F_0-\imath)*\eta_{\delta} \|_{H^{l}}+ \|\imath(x)*\eta_{\delta} -\imath\|_{H^{l}}\\
&=\|(F_0-\imath)*\eta_{\delta} \|_{H^{l}}\le C_n \|F_0\|_{H^{l}_{E}},
\end{align*}
where in the last equality we used
\begin{align*}
\int_{\Bbb R^n}y\eta_{\delta}(y)dy=0,\mbox{ } \int_{\Bbb R^n}x\eta_{\delta}(y)dy=x.
\end{align*}
Hence (\ref{8ujj}) has been done.

Given $\vec{\alpha}\in \Bbb R^{n}$  with $|\vec{\alpha}|=1$, by (\ref{y1099}) there must exist  some $\beta>0$ independent of $x,\vec{\alpha}$ such that  $|dF_0(x)(\vec{\alpha})|>\beta$ for all $x\in \Bbb R^n$.  Meanwhile, by the support of $\eta_{\delta}$ and $\eta_{\delta}*x=x$, one has
\begin{align*}
1_{|x|\ge 2R}(x)[dF^{\delta}_0(x)-d\imath]=1_{|x|\ge 2R}(x)\int_{\Bbb R^n} \eta_{\delta}(y)1_{|x-y|\ge  R}(y) (dF_0-d\imath)(x-y)dy
\end{align*}
for $R\ge 4$. Thus, for $R$ sufficiently large, (\ref{0tt11}) shows
\begin{align*}
\sup_{\delta\in (0,1)}\|dF^{\delta}_0(x)-d\imath\|_{L^{\infty}_x(|x|\ge R)}\le \frac{1}{10},
\end{align*}
where $\imath(x)=(x,0,0)$. In the ball $B_R:=\{x\in \Bbb R^n:|x|\le R\}$, the continuity of $dF_0$  implies that   $dF^{\delta}_0(\vec{\alpha})$ converges to $ dF_0(\vec{\alpha})$ uniformly for $x\in B_{R}$ and all $\alpha\in \Bbb R^n$ of unit length. Combining these two facts, we conclude that there exists $\delta_*>0$ sufficiently small and
some $\beta'>0$ so that
$$
\inf_{\delta\in (0,\delta_*)}\inf_{x\in\Bbb R^n} \min_{\vec{\alpha}\in \Bbb R^{n}} |dF^{\delta}_0(\vec{\alpha})|\ge \beta' >0.
$$
Therefore, we have verified  (\ref{9ujj}).

We also need the following lemma.
\begin{lemma}\label{yuhj000}
Let $n\ge 2$, $\gamma_n=2[\frac{n}{2}]+2$, $l\in \Bbb Z_+$ with $l\ge [\frac{n}{2}]+2$. If the immersion  $F_0\in C^{l}(\Bbb R^n;\Bbb R^{n+2})$ and $F_0\in H^{l}_{E}$, then  the associated $F^{\delta}_0$, second fundamental form ${\bf A}^{\delta}_0$ satisfy
 \begin{align*}
 \sup_{\delta\in (0,\delta_*)}\|{\bf A}^{\delta}_0\|_{ H^{l-2,\infty} }+ \sup_{\delta\in (0,\delta_*)}\|{\bf A}^{\delta}_0\|_{ H^{l-2,2} }+ \sup_{\delta\in (0,\delta_*)}\|{\bf A}^{\delta}_0\|_{ H^{0,\gamma_n} }\le  C(F_0).
 \end{align*}
\end{lemma}
\begin{proof}
First, it is easy to verify  $F^{\delta}_0\in C^{l}(\Bbb R^n;\Bbb R^{n+2})$ with bounds independent of $\delta$. Denote ${ }^{\delta}g$ the corresponding induced metric to $F^{\delta}_0$.
Second, the derivatives  $|\partial^{m}_x\mbox{ }{ }^{\delta}g_{ij}|$ are uniformly bounded in $x\in \Bbb R^n$ for all $0\le m\le l-1$ independent of $\delta$. As before using the $\delta$ independent   bound of the smallest eigenvalue of $({ }^{\delta}g_{ij})$, the derivatives  $|\partial^{m}_x\mbox{ }{ }^{\delta}g^{ij}|$ are uniformly bounded in $x\in \Bbb R^n$ for all $0\le m\le l-1$. Third,
writing   $\nabla^{l}{\bf A}^{\delta}$ via $\partial^{m}_xF^{\delta}_0$, ${ }^{\delta}g $ and $\partial^{j}_x\mbox{ }{ }^{\delta}\Gamma$ yields  uniform bounds of $\|{\bf A}^{\delta}_0\|_{H^{l-2,\infty}}$.
Finally, with these $C^{j}$ norms in hand, we get the desired  $L^2_x$ norms by  $F_0\in H^{l}_{E}$. The $H^{0,\gamma_n}$ bound  follows by H\"older inequality.
\end{proof}

{\it Now, assume that $F_0$ is a given immersion  fulfilling the conditions in Corollary \ref{z222}.}

{\bf Step 1} For each $\delta\in (0,\delta_*)$, since $F^{\delta}_{0}\in H^{\infty}_{E}$, Theorem 7.1 yields a smooth solution $F^{\delta}(x,t)$ to SMCF in $C([-T_{\delta},T_{\delta}];H^{\infty}_{E})$. By Lemma \ref{yuhj000}  and the fact that $T_{\delta}>0$ can be chosen to depend only on $n,\|{\bf A}^{\delta}_0\|_{H^{[\frac{n}{2}]+1}},\|{\bf A}^{\delta}_0\|_{L^{\gamma_n}_x}$, we see there exists $T>0$ independent of $\delta\in (0,\delta_*)$ such that $F^{\delta}(x,t)$ is smooth w.r.t $x\in \Bbb R^n$ for all $t\in [-T,T]$. Moreover, Theorem 7.1 and Lemma \ref{yuhj000}  give
 \begin{align}\label{klkl}
 \sup_{\delta\in (0,\delta_*)}[\|{\bf A}^{\delta}(t)\|_{C([-T,T]; H^{0,\infty} )} + \|{\bf A}^{\delta}(t)\|_{C([-T,T]; H^{s-2,2} ) }]\le  C(F_0),
 \end{align}
 which combined with Hamilton's interpolation theorem (e.g. Lemma 6.4)  implies there exists some  $p>n$
 \begin{align}\label{78pp}
 \sup_{\delta\in (0,\delta_*)} \|{\bf A}^{\delta}(t)\|_{C([-T,T]; H^{m,p}) } \le  C(F_0),\mbox{ }\forall 0\le m <\frac{2s-4}{n}.
 \end{align}
Using the point-wise equations
 \begin{align*}
 |\partial_t F^{\delta}|&\le \|{\bf H}^{\delta}\|_{L^{\infty}_{x}}\\
  |\partial_t \mbox{ }^{\delta}g_{ij}|&\le \|{\bf A}^{\delta}\|^2_{H^{0,\infty} }|{}^{\delta}g_{ij}|,
 \end{align*}
(see (\ref{7.7x})) we deduce that
 \begin{align}\label{4rf}
\sup_{\delta\in (0,\delta_*)}\|F^{\delta}\|_{C([-T,T]; C^1(\Bbb R^{n})) }\le   C(F_0).
 \end{align}
By Lemma \ref{se4},   (\ref{78pp}) and the conservation of local volume along the SMCF, we further obtain
 \begin{align}\label{88pp}
 \sup_{\delta\in (0,\delta_*)} \|{\bf A}^{\delta}(t)\|_{C([-T,T]; H^{j,\infty}) } \le  C(F_0),\mbox{ }\forall 0\le j<\frac{2s-4}{n}-1.
 \end{align}
{\bf Step 2.} From Lemma \ref{3Hami} and (\ref{klkl}), the metrics ${}^{\delta}g(t)$ are equivalent up to a uniform ratio constant depending only on $F_0,n$.
By (\ref{5r}), we get
 \begin{align}\label{90pp}
 \sup_{\delta\in (0,\delta_*)}\|\partial_t \mbox{ }^{\delta}\Gamma\|_{H^{j,\infty} }\le   C(F_0),\mbox{ }\forall 0\le j<\frac{2s-4}{n}-2.
\end{align}
Integrating  (\ref{5r}) with respect to $t$, one has
 \begin{align}\label{91pp}
 \sup_{\delta\in (0,\delta_*)}\| { }^{\delta}\Gamma\|_{H^{0,\infty}}\le     C(F_0).
\end{align}
Meanwhile, similar to (\ref{jnkmL}), for any tensor $\mathbb{T}$, one has
\begin{align}\label{92pp}
\|\partial^{l} \mathbb{T}\|_{L^{\infty}_{x}}\le C(F_0)\|\nabla^l \mathbb{T}\|_{L^{\infty}_x}+C(F_0)(\sum^{l-1}_{m=0}\|\partial^{m}\Gamma\|_{ L^{\infty}_x}  +\|\partial^{m}\mathbb{T}\|_{  L^{\infty}_x}).
\end{align}
By (\ref{90pp}), (\ref{91pp}), (\ref{92pp}) and induction, we
obtain
\begin{align*}
 \sup_{\delta\in (0,\delta_*)}\|\partial^{j}(\partial_t\mbox{ }^{\delta}\Gamma) \|_{C([-T,T];C^0(\Bbb R^n))}&\le   C( F_0),\mbox{ }\forall 0\le j<\frac{2s-4}{n}-2 \\
 \sup_{\delta\in (0,\delta_*)}\|\partial^{j}\mbox{ }^{\delta}\Gamma\|_{C([-T,T];C^0(\Bbb R^n))}&\le   C(F_0),\mbox{ }\forall 0\le j<\frac{2s-4}{n}-2.
\end{align*}
Then by  (\ref{4rf}), the two identities in (\ref{56tgvbbn}) and induction, one deduces
\begin{align*}
 \sup_{\delta\in (0,\delta_*)}\|\partial^{2}_xF^{\delta}\|_{C([-T,T];C^j(\Bbb R^n))}&\le  C( F_0), \mbox{ }\forall  0\le j<\frac{2s-4}{n}-2,\\
  \sup_{\delta\in (0,\delta_*)}\| F^{\delta}\|_{C([-T,T];C^1(\Bbb R^n))}&\le C( F_0).
\end{align*}
Therefore, we conclude
\begin{align}
 \sup_{\delta\in (0,\delta_*)}\|F^{\delta}\|_{C([-T,T];C^l(\Bbb R^n))}\le    C_l( F_0),\mbox{ }\forall \mbox{ } 0\le l<\frac{2s-4}{n}.\label{11z}
\end{align}
It follows from (\ref{11z}) that there exists a sequence $\delta_{m}\to 0$ such that $F^{\delta_m}$ converges to $F$ in $C([-T,T];C^{\beta}(\Bbb R^n))$ for any  $\beta\in [0,\beta_s)$  with $\beta_s= [\frac{2s-4}{n}]-1$. Since $s\ge 2n+2$, $F^{\delta_m}$ must converge in $C^2$. Thus $F \in C([-T,T];C^{\beta_s}(\Bbb R^n))$  solves SMCF point-wisely.
The continuous dependence claimed in  Corollary \ref{z222} follows by  \cite{SGGoGGnGGgGG2}. Therefore, Corollary \ref{z222} has been proved.

\section*{Acknowledgement}
The author owes sincere gratitude  to the referees for the insightful comments which largely improved the presentation of this work.
 This work is partially supported by NSF-China Grant-1200010237 and Grant-11631007.

\end{document}